\documentclass[11pt,reqno]{amsart}
\usepackage{amsmath,amssymb,mathrsfs}
\usepackage{graphicx,cite}

\newtheorem{theorem}{Theorem}[section]

\newtheorem{lemma}[theorem]{Lemma}

\newtheorem{assumption}{Assumption}
\newtheorem{definition}{Definition}

\numberwithin{equation}{section}

\begin{document}

\title[an inverse random source problem]{An inverse random source
problem for the time fractional diffusion equation driven by a fractional
Brownian motion}

\author{Xiaoli Feng}
\address{School of Mathematics and Statistics, Xidian University, Xi'an, 713200,
P. R. China}
\email{xiaolifeng@xidian.edu.cn}

\author{Peijun Li}
\address{Department of Mathematics, Purdue University, West Lafayette, Indiana 47907, USA}
\email{lipeijun@math.purdue.edu}

\author{Xu Wang}
\address{Department of Mathematics, Purdue University, West Lafayette, Indiana 47907, USA}
\email{wang4191@purdue.edu}

\thanks{The research is supported in by part the NSF grant DMS-1912704.}

\subjclass[2010]{35R30, 35R60, 65M32}

\keywords{fractional diffusion equation, inverse source problem, fractional
Brownian motion, uniqueness, ill-posedness}

\begin{abstract}
This paper is concerned with the mathematical analysis of the inverse
random source problem for the time fractional diffusion equation, where the
source is assumed to be driven by a fractional Brownian motion. Given the
random source, the direct problem is to study the stochastic time fractional
diffusion equation. The inverse problem is to determine the statistical
properties of the source from the expectation and variance of the final time
data. For the direct problem, we show that it is well-posed and has a unique
mild solution under a certain condition. For the inverse problem, the uniqueness
is proved and the instability is characterized. The major ingredients of the
analysis are based on the properties of the Mittag--Leffler function and the
stochastic integrals associated with the fractional Brownian motion.
\end{abstract}

\maketitle

\section{Introduction}

In the last two decades, the fractional derivative equations (FDEs) have
received ever-increasing attention by many researchers due to their potential 
applications in modeling real physical phenomena. For examples, the FDE can be
used to describe the anomalous diffusion in a highly heterogeneous aquifer
\cite{Adams+1992}, the relaxation phenomena in complex viscoelastic materials
\cite{Ginoa+1992}, the anomalous diffusion in an underground environmental
problem \cite{Hatano+1998}, and a non-Markovian diffusion process with memory
\cite{Metzler+2000}. We refer to \cite{Gorenflo09} for some recent advances in
theory and simulation of the fractional diffusion processes.

Motivated by significant scientific and industrial applications, the field of
inverse problems has undergone a tremendous growth in the last several decades
since Calder\'{o}n proposed an inverse conductivity problem. Recently, the
inverse problems on FDEs have also progressed into an area of intense research
activity. In particular, for the time or time-space fractional diffusion
equations, the inverse source problems have been widely  investigated
mathematically and numerically. Compared with the semilinear problem
\cite{Luchko+2013}, many more results are available for the linear problems. The
linear inverse source problems for fractional diffusion equations can be broadly
classified into the following six cases: (1) determining a space-dependent
source term from the space-dependent data
\cite{Aziz+2016,Furati+2014,Kirane+2011,Kirane+2013,Sakamoto+2011+IS,
Tatar+Tina+2015,Tatar+2015,Tuan+2017,Wang+2013,Wang+Yamamoto+Han+2013,Wei+2014,
Yamamoto+2012,Yang+2017}; (2) determining a time-dependent source term
from the time-dependent data
\cite{Ismailov+2016,Li+Wei+2018,Liu+Zhang+2017,Sakamoto+Yamamoto+2011,
Wei+Zhang+2013}; (3) determining a time-dependent source term from the
space-dependent data \cite{Aleroev+2013,Jia+Peng+2017}; (4) determining a
space-dependent source term from the time-dependent data \cite{Zhang+Xu+2011};
(5) determining a space-dependent source term from the boundary data
\cite{Wei+Sun+2016}; (6) determining a general source from the
time-dependent data \cite{Murio+Mejia+2008}. Despite a considerable amount of
work done so far, the rigorous mathematical theory is still lacking
\cite{Jin+2015}, especially for the inverse problems where the sources contain
uncertainties, which are known as the inverse random source problems.

The inverse random source problems belong to a category of stochastic inverse
problems, which refer to inverse problems that involve uncertainties. Compared
to deterministic inverse problems, stochastic inverse problems have
substantially more difficulties on top of the existing obstacles due to the
randomness and uncertainties. There are some work done for the inverse random
source scattering problems, where the wave propagation is governed by the
stochastic Helmholtz equation driven by the white noise. In
\cite{Devaney+1979}, it was shown that the correlation of the random source
could be determined uniquely by the correlation of the random wave field.
Recently, an effective computational model was developed in
\cite{Bao+Chen+Li+2016,Bao+Chen+Li+2017,Bao+Chow+Li+Zhou+2010,
Bao+Chow+Li+Zhou+2014,Li+Chen+Li+2017,Li+2011,Li+Yuan+2017}, the goal was to
reconstruct the statistical properties of the random source such as
the mean and variance from the boundary measurement of the radiated random
wave field at multiple frequencies.

The work is very rare for the inverse random source problems of the
fractional diffusion equations. In \cite{Niu+2018}, the authors presented a
study on the random source problem for the fractional diffusion equation.
Specifically, they considered the following initial-boundary value problem:
\begin{equation}\label{fBmIsource}
\left\{
\begin{array}{ll}
\partial_{t}^{\alpha}u(x,t)-\Delta u(x,t)=f(x)h(t)+g(x)\dot{W}(t), &
(x,t)\in{D}\times(0,T),\\
u(x,t)=0, & (x,t)\in{\partial D}\times[0,T], \\
u(x,0)=0, & x\in\overline{{D}},
\end{array}
\right.
\end{equation}
where $D$ is a bounded domain with the Lipschitz boundary $\partial D$,
$f$ and $g$ are deterministic functions with compact supports contained in
$D$, $h$ is also a deterministic function, $W$ and $\dot{W}$ are the
Brownian motion and the white noise, respectively, and
$\partial_{t}^{\alpha}u(x,t), 0<\alpha\leq 1$ is the Caputo fractional
derivative given by
\[
\partial_{t}^{\alpha}u(x,t)=
\begin{cases}
\displaystyle  \frac{1}{\Gamma(1-\alpha)}\int_0^t\frac{\partial u(x,s)}{\partial
s}\frac{ds}{(t-s)^{\alpha}},& \quad 0<\alpha<1,\\
\partial_t u(t, x), &\quad \alpha=1.
\end{cases}
\]
Here $\Gamma(\alpha)=\int_0^\infty e^{-s}s^{\alpha-1}ds$ is the Gamma function.
For the model problem \eqref{fBmIsource}, the authors studied the inverse
problem of reconstructing $f(x)$ and $|g(x)|$ from the statistics of the final
time data $u(x,T)$ with $\frac{1}{2}<\alpha<1$. It was shown that $f$ and $|g|$
can be uniquely determined by the expectation and covariance of the final
data, respectively. Besides, they also showed that the inverse problem is not
stable in the sense that a small variance of the data may lead to a huge error
of the reconstruction. Naturally, one may ask the following two questions: 
\begin{itemize}
  \item []Q1. Can the results be extended to  $0<\alpha<1$ for the Brownian
motion?
  \item []Q2. Can the results be extended to the fractional Brownian motion?
\end{itemize}

Motivated by above reasons, the main purpose of this paper is to study the
inverse source problem for the time fractional diffusion equation, where the
source is assumed to be driven by a more general stochastic process:
the fractional Brownian motion $B^H(t)$ with $0<\alpha\leq 1$, $0<H<1$, where
$H$
is called the Hurst index of the fractional Brownian motion. Clearly, the
model equation (\ref{fBmIsource}) is reduced to the classical heat conduction
equation with the Brownian motion for $\alpha=1$. In this work, we give
confirmative answers to Q1 and Q2. For Q1, due to the singular integral (see
Lemma 3.4 in \cite{Niu+2018} or the proof later in this paper), the results can
not be extended; for Q2, the results can be extended as long as $\alpha+H>1$.
For the restriction $\alpha+H>1$, it is not difficult to understand since both
$H$ and $\alpha$ imply some smoothness requirement of the solution for
the model equation.

The rest of this paper is organized as follows. In Section 2, we introduce some
preliminaries for the time-fractional diffusion equations and the
Mittag--Leffler function. Section 3 is
concerned with the well-posedness of the direct problem. Section 4 is devoted
to the inverse problem. The two cases $0<H<\frac12$ and $\frac12<H<1$ are 
discussed separately for both of the direct and inverse problems. The paper is
concluded with some general remarks and directions for future research in
Section 5. To make the paper easily accessible, some necessary notations and
useful results are provided in Appendix on the fractional Brownian motion.

\section{Preliminaries}

Let the triple $(\Omega, \mathcal F, P)$ be a complete probability space on
which the fractional Brownian motion $B^H$ is defined (see Appendix for
the details). Here $\Omega$ is a sample space, $\mathcal F$ is a
$\sigma$-algebra on $\Omega$, and $P$ is a probability measure on the measurable
space $(\Omega, \mathcal F)$. If $X$ is a random variable, $\mathbb{E}(X)$ and
$\mathbb{V}(X)=\mathbb{E}(X-\mathbb{E}(X))^2 =
\mathbb{E}(X^2)-(\mathbb{E}(X))^2$ are the expectation and variance of $X$,
respectively. If $X, Y$ are two random
variables, $\text{Cov}(X,Y)=\mathbb{E}[(X-\mathbb{E}(X))(Y-\mathbb{E}(Y))]$
denotes the covariance of $X$ and $Y$.

Consider initial-boundary value problem of the fractional diffusion equation
with a random source driven by the fractional Brownian motion
\begin{equation}\label{fde}
\left\{
\begin{array}{ll}
\partial_{t}^{\alpha}u(x,t)-\Delta u(x,t)=f(x)h(t)+g(x)\dot{B}^H(t), &
(x,t)\in{D}\times(0,T),\\
u(x,t)=0, & (x,t)\in{\partial D}\times[0,T], \\
u(x,0)=0, & x\in\overline{{D}}.
\end{array}
\right.
\end{equation}
Let $\{\lambda_k, \varphi_k\}_{k=1}^{\infty}$ be the eigensystem of the operator
$-\Delta$ with the homogeneous Dirichlet boundary condition in $D$. It is known
that the eigenvalues satisfy
$0<\lambda_1\leq\lambda_2\leq\cdots\leq\lambda_k\leq\cdots,
\lambda_k\to\infty, k\to\infty$ and the
eigen-functions $\{\varphi_k\}_{k=1}^{\infty}$ form a complete and orthogonal
basis in $L^2({D})$. It follows from the separation of variables that the
solution of \eqref{fde} can be written as
\[
 u(x,t,\omega)=\sum_{k=1}^{\infty}u_k(t,\omega)\varphi_k(x),
\]
where $\omega\in\Omega$ and $u_k(t,\omega)$ satisfies the stochastic fractional
differential equation
 \begin{equation}\label{SODE}
\left\{
\begin{array}{ll}
D_t^{\alpha}u_k(t,\omega)+ \lambda_k
u_k(t,\omega)=f_kh(t)+g_k\dot{B}^H(t), &  t\in(0,T),\\
u_k(0)=0.
\end{array}
\right.
\end{equation}
Here $f_k=(f,\varphi_k)_{L^2({D})}$ and
$g_k=(g,\varphi_k)_{L^2({D})}$. When
$g_k=0, k\in\mathbb N$, the corresponding deterministic fractional differential
equation is
\begin{equation*}
\left\{
\begin{array}{ll}
D_t^{\alpha}u_k(t)+ \lambda_k
u_k(t)=f_kh(t), &  t\in(0,T),\\
u_k(0)=0,
\end{array}
\right.
\end{equation*}
whose solution can be obtained directly by applying the following Lemma. The
proof can be found in \cite[Page 230]{Kilbas+2006} or \cite[Example
4.3]{Podlubny+1999}.

\begin{lemma}\label{fODEsolution}
Consider the Cauchy problem for the fractional differential equation:
 \begin{equation}\label{fODE}
\begin{cases}
D_t^{\alpha}v(t) - \lambda v(t)=f(t), &  t\in(0,T),\\
\frac{d^{n}v}{d t^{n}}(0)=v_n, & n=0,\ldots,\lfloor\alpha\rfloor.
\end{cases}
\end{equation}
If $f(t)\in C^{0, \gamma}$ with $0\leq\gamma\leq\alpha$, then the
Cauchy problem (\ref{fODE}) has a unique solution given by
\begin{equation*}
v(t)=\sum_{n=0}^{\lfloor\alpha\rfloor}v_nt^nE_{\alpha,n+1}(\lambda
t^{\alpha})+\int_0^t(t-\tau)^{\alpha-1}E_{\alpha,\alpha}(\lambda(t-\tau)^{\alpha
})f(\tau)d\tau,
\end{equation*}
where $E_{\alpha,\beta}$ is the Mittag--Leffler function (see (\ref{MLf})).
\end{lemma}

By Lemma \ref{fODEsolution}, we can obtain a mild solution of
(\ref{SODE}), which gives a mild solution to the initial-boundary value
problem of the stochastic fractional diffusion equation (\ref{fBmIsource}). Let
us first give some assumptions in order to understand the solution.

\begin{assumption}\label{assumption}
Assume that $f,g\in L^2({D})$ such that $g\neq0$ and $h\in L^{\infty}(0,T)$
is positive and bounded from below, i.e., there exists $c_h>0$ such that $h\geq
c_h$.
\end{assumption}

\begin{definition}
A stochastic process $u: D\times [0,T]\rightarrow L^2({D})$ defined
by
\begin{equation}\label{Dsolution}
u(x,t,\omega)=\sum_{k=1}^{\infty}(I_{k,1}(t)+I_{k,2}(t,\omega))\varphi_k(x),
\end{equation}
is called a mild solution of the initial-boundary value problem of the
stochastic fractional diffusion equation (\ref{fBmIsource}), where
\begin{align}\label{Ik1}
I_{k,1}(t)&=f_k\int_0^t(t-\tau)^{\alpha-1}E_{\alpha,\alpha}(-\lambda_k(t-\tau)^{
\alpha})h(\tau)d\tau,\\
\label{Ik2}
I_{k,2}(t,\omega)&=g_k\int_0^t(t-\tau)^{\alpha-1}E_{\alpha,\alpha}
(-\lambda_k(t-\tau)^{\alpha})dB^H(\tau).
\end{align}
\end{definition}

Since the Mittag--Leffler function is very important for the analysis, let us
state some of its properties. The two-parametric Mittag--Leffler
function is defined as
\begin{equation}\label{MLf}
E_{\alpha,\beta}(z)=\sum_{k=0}^{\infty}\frac{z^k}{\Gamma(k\alpha+\beta)}, \quad
z\in \mathbb{C},
\end{equation}
where $\alpha, \beta\in\mathbb R$. Obviously, $E_{1,1}(z)=e^{z}$. More
information about the Mittag--Leffler
function can be found in \cite{MLfunction}.

\begin{lemma}\cite[Theorem 1.6]{Podlubny+1999}\label{MLinq}
If $0<\alpha<2$, $\beta$ is an arbitrary real number, $\mu$ is such that
$\pi\alpha/2<\mu<\min\{\pi,\pi\alpha\}$, then there exists a positive constant
$C$ such that 
\begin{equation*}
|E_{\alpha,\beta}(z)|\leq\frac{C}{1+|z|},\quad\mu\leq|arg(z)|\leq\pi,\quad
|z|\geq 0.
\end{equation*}
\end{lemma}

\begin{lemma}\cite[Lemma 3.2]{Sakamoto+Yamamoto+2011}\label{dEalf}
For $\lambda>0, \alpha>0$, we have
\begin{equation*}
\frac{d}{dt}E_{\alpha,1}(-\lambda t^{\alpha})=-\lambda
t^{\alpha-1}E_{\alpha,\alpha}(-\lambda t^{\alpha}),\quad t>0.
\end{equation*}
\end{lemma}

\begin{lemma}\label{dML2} For $\lambda, z\in \mathbb{C}$, we have
\begin{equation*}
\frac{d}{dz}(z^{\alpha-1}E_{\alpha,\alpha}(-\lambda
z^{\alpha}))=z^{\alpha-2}E_{\alpha,\alpha-1}(-\lambda z^{\alpha}).
\end{equation*}
\end{lemma}

\begin{proof}
By \cite[formula (4.3.1)]{MLfunction}
\begin{equation*}
\frac{d}{dz}(z^{\alpha-1}E_{\alpha,\alpha}(
z^{\alpha}))=z^{\alpha-2}E_{\alpha,\alpha-1}(z^{\alpha}),
\end{equation*}
which completes the proof after using the chain rule.
\end{proof}

\begin{lemma}\label{lm:tE}
For any $0<s<t, \lambda_k>0$, there exists some constant $C$ such that
\begin{align*}
|t^{\alpha-1}E_{\alpha,\alpha}(-\lambda_kt^{\alpha})-s^{\alpha-1}E_{\alpha,
\alpha}(-\lambda_ks^{\alpha})|
\le C\int_s^t\frac{r^{\alpha-2}}{1+\lambda_kr^{\alpha}}dr.
\end{align*}
\end{lemma}

\begin{proof}
By Lemmas \ref{dML2} and \ref{MLinq}, we have
\[
\frac
d{dr}[r^{\alpha-1}E_{\alpha,\alpha}(-\lambda_kr^{\alpha})]=r^{\alpha-2}E_{\alpha
,\alpha-1}(-\lambda_kr^{\alpha})
\]
and
\[
|E_{\alpha,\alpha-1}(-\lambda_kr^{\alpha})|\le\frac{C}{1+\lambda_kr^{\alpha}}.
\]
A simple calculation yields that
\begin{align*}
|t^{\alpha-1}E_{\alpha,\alpha}(-\lambda_kt^{\alpha})-s^{\alpha-1}E_{\alpha,
\alpha}(-\lambda_ks^{\alpha})|
=&\int_s^tr^{\alpha-2}E_{\alpha,\alpha-1}(-\lambda_kr^{\alpha})dr\\
\le&C\int_s^t\frac{r^{\alpha-2}}{1+\lambda_kr^{\alpha}}dr,
\end{align*}
which completes the proof.
\end{proof}

\begin{lemma}\cite{Pollard+1984}\label{c.m.}
For $x\geq 0, 0\leq\alpha\leq 1$, the function $E_{\alpha,1}$ is completely
monotonic, i.e.,
\begin{equation*}
(-1)^n\frac{d^nE_{\alpha,1}(-x)}{dx^n}\geq0,\quad n=0,1,2,\cdots.
\end{equation*}
\end{lemma}

By Lemmas \ref{dEalf} and \ref{c.m.}, we have the following property
of $E_{\alpha,\alpha}$.

\begin{lemma}\label{Epositive}
For $0<\alpha\leq 1, x\geq0$, there holds $E_{\alpha,\alpha}(-x)\geq0$ and
$x^{\alpha-1}E_{\alpha,\alpha}(-\lambda x^{\alpha})$ is monotonically
decreasing.
\end{lemma}

\section{The direct problem}

In this section, we discuss the well-posedness of the direct problem. We show
that the mild solution \eqref{Dsolution} is well-defined for the
initial-boundary value problem of the stochastic fractional diffusion equation
(\ref{fBmIsource}).

It is easy to note that the mild solution (\ref{Dsolution}) satisfies
\begin{align*}
  \|u(\cdot,t)\|^2_{L^2({D})}
&=\left\|\sum_{k=1}^{\infty}(I_{k,1}(t)+I_{k,2}(t,\omega))\varphi_k(\cdot)
\right\|_{L^2({D})}^2 \\
&=\sum_{k=1}^{\infty}(I_{k,1}(t)+I_{k,2}(t,\omega))^2\leq 2
\left(\sum_{k=1}^{\infty}
I^2_{k,1}(t)+\sum_{k=1}^{\infty}I_{k,2}^2(t,\omega)\right).
\end{align*}
Hence,
\begin{align}\label{ex}
\mathbb{E}\left[\|u\|^2_{L^2({D}\times[0,T])}\right] & =
\mathbb{E}\left[\int_0^T \|u(\cdot,t)\|^2_{L^2({D})}dt\right]\nonumber\\
&\lesssim\mathbb{E}\left[\int_0^T
\left(\sum_{k=1}^{\infty}I^2_{k,1}(t)+\sum_{k=1}^{\infty}I_{k,2}^2(t,
\omega)\right)dt\right]\nonumber\\
&=\int_0^T
\left(\sum_{k=1}^{\infty}I^2_{k,1}(t)\right)dt+\mathbb{E}\left[\int_0^T
\left(\sum_{k=1}^{\infty}I_{k,2}^2(t,\omega)\right)dt\right]\nonumber\\
&=\sum_{k=1}^{\infty}\|I_{k,1}\|^2_{L^2(0,T)}+\int_0^T
\left(\sum_{k=1}^{\infty}\mathbb{E}\left[I_{k,2}^2(t,\omega)\right]
\right)dt\nonumber\\
&=:S_1+S_2.
\end{align}
Hereinafter $a\lesssim b$ stands for $a\leqslant Cb$, where $C>0$ is a
constant.

We shall discuss the sums $S_1$ and $S_2$ separately. First, let us
consider the sum $S_1$. Set
$G_{\alpha,k}(t)=t^{\alpha-1}E_{\alpha,\alpha}(-\lambda_kt^{\alpha})$. By
(\ref{Ik1}), it is easy to see that
$I_{k,1}(t)=f_k\left(G_{\alpha,k}*h\right)(t)$. Using the Young convolution
inequality yields
\begin{equation}\label{Ik11}
\|I_{k,1}\|_{L^2(0,T)}\leq |f_k|\|G_{\alpha,k}\|_{L^1(0,T)}\|h\|_{L^2(0,T)}.
\end{equation}
It follows from Lemma \ref{MLinq} that
\begin{equation}\label{G}
\|G_{\alpha,k}\|_{L^1(0,T)}=\int_0^T|t^{\alpha-1}E_{\alpha,\alpha}(-\lambda_kt^{
\alpha})|dt\lesssim\int_0^Tt^{\alpha-1}dt=\frac{T}{\alpha}.
\end{equation}
Combining (\ref{ex})--(\ref{G}), we obtain
\begin{equation}\label{S1}
S_1\leq\frac{T^2}{\alpha^2}\sum_{k=1}^{\infty}|f_k|^2\|h\|^2_{L^2(0,T)}
\lesssim\|h\|^2_{L^2(0,T)}\|f\|^2_{L^2({D})}.
\end{equation}

Next, we estimate the sum $S_2$. By (\ref{Ik2}), we know that
\begin{align}\label{exIk2}
\mathbb{E}\left[I_{k,2}^2(t,\omega)\right]&=\mathbb{E}\left[
g_k^2\left(\int_0^t(t-\tau)^{\alpha-1}E_{\alpha,\alpha}(-\lambda_k(t-\tau)^{
\alpha})dB^H(\tau)\right)^2\right]\nonumber\\
&=g_k^2
\mathbb{E}\left[\left(\int_0^t(t-\tau)^{\alpha-1}E_{\alpha,\alpha}
(-\lambda_k(t-\tau)^{\alpha})dB^H(\tau)\right)^2\right].
\end{align}
The case $H=\frac{1}{2}, \alpha\in(\frac12,1)$ has been considered in
\cite{Niu+2018}. We investigate more general $\alpha\in (0,1)$, and
discuss the cases $H\in(0,\frac{1}{2})$ and $H\in(\frac{1}{2},1)$,
respectively, since the stochastic integrals are different.

\subsection{The case $H\in(0,\frac{1}{2})$}\label{subsec3.1}

It follows from Appendix on the fractional Brownian motion $B^H$ that the
stochastic integral in (\ref{exIk2}) with respect to $B^H$ satisfies
\begin{align}\label{eq:moment}
&\mathbb{E}\left|\int_0^t(t-\tau)^{\alpha-1}E_{\alpha,\alpha}
(-\lambda_k(t-\tau)^{\alpha})dB^H(\tau)\right|^2\nonumber\\
=&\int_0^t\left[K_{H,t}^*\left((t-\cdot)^{\alpha-1}E_{\alpha,\alpha}
(-\lambda_k(t-\cdot)^{\alpha})\right)\right]^2(\tau)d\tau\nonumber\\
=&\int_0^t\Bigg[K_H(t,\tau)(t-\tau)^{\alpha-1}E_{\alpha,\alpha}
(-\lambda_k(t-\tau)^{\alpha})\nonumber\\
&+\int_{\tau}^t\left[(t-u)^{\alpha-1}E_{\alpha,\alpha}(-\lambda_k(t-u)^{\alpha}
)-(t-\tau)^{\alpha-1}E_{\alpha,\alpha}(-\lambda_k(t-\tau)^{\alpha})\right]\frac{
\partial K_H(u,\tau)}{\partial u}du\Bigg]^2d\tau\nonumber\\
\lesssim&\int_0^t\left[\left(\frac
{t}{\tau}\right)^{H-\frac12}(t-\tau)^{\alpha+H-\frac32}E_{\alpha,\alpha}
(-\lambda_k(t-\tau)^{\alpha})\right]^2d\tau\nonumber\\
&+\int_0^t{\tau}^{1-2H}\left[\left(\int_\tau^tu^{H-\frac32}(u-\tau)^{H-\frac12}
du\right)(t-\tau)^{\alpha-1}E_{\alpha,\alpha}(-\lambda_k(t-\tau)^{\alpha})\right
]^2d\tau\nonumber\\
&+\int_0^t\left[\int_\tau^t\left[(t-u)^{\alpha-1}E_{\alpha,\alpha}
(-\lambda_k(t-u)^{\alpha})-(t-\tau)^{\alpha-1}E_{\alpha,\alpha}
(-\lambda_k(t-\tau)^{\alpha})\right]\frac{\partial K_H(u,\tau)}{\partial
u}du\right]^2d\tau\nonumber\\
=&:I_1(t)+I_2(t)+I_3(t),
\end{align}
where $K_H(u,\tau)$ is given by (\ref{KH012}). Below we estimate $I_j(t), j=1,
2, 3.$

The estimate of $I_1(t)$. By Lemma \ref{MLinq}, there holds
\begin{align}\label{eq:It}
I_1(t)=&\int_0^t\left(\frac
{t}{\tau}\right)^{2H-1}(t-\tau)^{2\alpha+2H-3}E_{\alpha,\alpha}
^2(-\lambda_k(t-\tau)^{\alpha})d\tau\nonumber\\
\lesssim&t^{2H-1}\int_0^t\tau^{1-2H}(t-\tau)^{2\alpha+2H-3}d\tau\nonumber\\
\leq&\int_0^t(t-\tau)^{2\alpha+2H-3}d\tau
=\frac{t^{2\alpha+2H-2}}{2\alpha+2H-2},
\end{align}
where we have used the conditions $0<H<\frac12, \alpha+H>1$ for the
singular integral and the mean value theorem for the definite integral.

The estimate of $I_2(t)$. Using Lemma \ref{MLinq}, we have
\begin{align}\label{eq:IIt}
I_2(t)=&\int_0^t{\tau}^{1-2H}\left(\int_\tau^tu^{H-\frac32}(u-\tau)^{H-\frac12}
du\right)^2(t-\tau)^{2\alpha-2}E_{\alpha,\alpha}^2(-\lambda_k(t-\tau)^{\alpha}
)d\tau\nonumber\\
\lesssim&\int_0^t{\tau}^{1-2H}\left(\int_\tau^tu^{H-\frac32}(u-\tau)^{H-\frac12}
du\right)^2(t-\tau)^{2\alpha-2}d\tau.
\end{align}
Since $H>0$, the integral $\int_\tau^tu^{H-\frac32}(u-\tau)^{H-\frac12}du$ is
well-defined. Furthermore, we have from the binomial expansion that
\begin{align*}
&\int_\tau^tu^{H-\frac32}(u-\tau)^{H-\frac12}du= \int_\tau^t
u^{2H-2}(1-\frac{\tau}{u})^{H-\frac12}du \\
=&\int_\tau^t u^{2H-2}\left[\sum_{n=0}^{\infty}\left(
\begin{array}{c}
H-\frac12\\
n
\end{array}\right)(-\frac{\tau}{u})^n\right]du \\
=&\sum_{n=0}^{\infty}\left(
\begin{array}{c}
H-\frac12\\
n
\end{array}\right)(-1)^n\tau^n\int_\tau^t u^{2H-2-n}du\\
=&\sum_{n=0}^{\infty}\left(
\begin{array}{c}
H-\frac12\\
n
\end{array}\right)(-1)^n\tau^n\frac{t^{2H-1-n}-\tau^{2H-1-n}}{2H-1-n}\\
=&t^{2H-1}\sum_{n=0}^{\infty}\left(
\begin{array}{c}
H-\frac12\\
n
\end{array}\right)\frac{(-1)^n}{2H-1-n}\left(\frac{\tau}{t}\right)^n-\tau^{2H-1}
\sum_{n=0}^{\infty}\left(
\begin{array}{c}
H-\frac12\\
n
\end{array}\right)\frac{(-1)^n}{2H-1-n}\\
\leq&(t^{2H-1}-\tau^{2H-1})\sum_{n=0}^{\infty}\left(
\begin{array}{c}
H-\frac12\\
n
\end{array}\right)\frac{(-1)^n}{2H-1-n}.
\end{align*}
It is easy to note from the asymptotic expansion for the binomial
coefficients that
\[
\sum_{n=0}^{\infty}\left(
\begin{array}{c}
H-\frac12\\
n
\end{array}\right)\frac{(-1)^n}{2H-1-n}=A<\infty.
\]
Therefore, (\ref{eq:IIt}) becomes
\begin{align*}
I_2(t)\lesssim&\int_0^t{\tau}^{1-2H}\left(t^{4H-2}+\tau^{4H-2}\right)(t-\tau)^{
2\alpha-2}d\tau\\
=&t^{4H-2}\int_0^t{\tau}^{1-2H}(t-\tau)^{2\alpha-2}d\tau+\int_0^t{\tau}^{2H-1}
(t-\tau)^{2\alpha-2}d\tau\\
=&t^{4H-2}t^{1-2H}\int_0^t(t-\tau)^{2\alpha-2}d\tau+\int_0^t{\tau}^{2H-1}t^{
2\alpha-2}(1-\frac{\tau}{t})^{2\alpha-2}d\tau\\
=&\frac{t^{2H+2\alpha-2}}{2\alpha-1}+t^{2\alpha-2}\int_0^t{\tau}^{2H-1}\left[
\sum_{n=0}^{\infty}\left(
\begin{array}{c}
2\alpha-2\\
n
\end{array}\right)\left(-\frac{\tau}{t}\right)^n\right]d\tau\\
\lesssim&t^{2H+2\alpha-2}+t^{2\alpha-2}\sum_{n=0}^{\infty}
\left[\left(\begin{array}{c}
2\alpha-2\\
n
\end{array}\right)(-1)^nt^{-n}\int_0^t{\tau}^{2H-1+n}d\tau\right]\\
=&t^{2H+2\alpha-2}+t^{2H+2\alpha-2}\sum_{n=0}^{\infty}\left(\begin{array}{c}
2\alpha-2\\
n
\end{array}\right)\frac{(-1)^n}{2H+n},
\end{align*}
where we have used the conditions $0<H<\frac12, \frac12<\alpha<1$.
Since $2\alpha-2<0$, we have from the asymptotic expansion for the
binomial coefficients again that
\[
0<\sum_{n=0}^{\infty}\left(
\begin{array}{c}
2\alpha-2\\
n
\end{array}\right)\frac{(-1)^n}{2H+n}=B<\infty.
\]
Hence
\begin{equation}\label{eq:IIt3}
I_2(t)\lesssim t^{2H+2\alpha-2}.\\
\end{equation}

The estimate of $I_3(t)$. Based on Lemma \ref{lm:tE}, for $0<\tau<u<t$, there
holds
\begin{align*}
&|(t-u)^{\alpha-1}E_{\alpha,\alpha}(-\lambda_k(t-u)^{\alpha})-(t-\tau)^{\alpha-1
}E_{\alpha,\alpha}(-\lambda_k(t-\tau)^{\alpha})|\\
\lesssim &\int_{t-u}^{t-\tau}r^{\alpha-2}dr
\lesssim (t-u)^{\alpha-\frac{3}{2}}(u-\tau)^{\frac{1}{2}},
\end{align*}
where we have used the fact that $x^{\frac{1}{2}}$ is $\frac12$-H\"{o}lder
continuous for $x>0$. A simple calculation yields that
\begin{align*}
I_3(t)\lesssim\int_0^t\left[\int_\tau^t(t-u)^{\alpha-\frac{3}{2}}(u-\tau)^{H-1}
\left(\frac{u}{\tau}\right)^{H-\frac12}du\right]^2d\tau.
\end{align*}
The above integral is convergent due to the conditions $H>0, \alpha>\frac12$.
Since $H\in(0,\frac12)$, we have $\left(\frac{u}{\tau}\right)^{H-\frac12}<1$ for
$0<\tau<u<t$. Hence
\begin{align}\label{eq:IIIt}
I_3(t)\lesssim&\int_0^t\left[\int_0^{t-\tau}(t-\tau-r)^{\alpha-\frac{3}{2}}r^{H-
1 }dr\right]^2d\tau\nonumber\\
=&\int_0^t\left[(t-\tau)^{\alpha-\frac{3}{2}}\sum_{n=0}^{\infty}\left(
\begin{array}{c}
\alpha-\frac{3}{2}\\
n
\end{array}\right)(-1)^n(t-\tau)^{-n}\int_0^{t-\tau}r^{n+H-1}dr\right]
^2d\tau\nonumber\\
=&\left[\sum_{n=0}^{\infty}\left(
\begin{array}{c}
\alpha-\frac32\\
n
\end{array}\right)\frac{(-1)^n}{n+H}\right]^2\int_0^t(t-\tau)^{2\alpha+2H-3}
d\tau \lesssim t^{2\alpha+2H-2},
\end{align}
where we have used the condition $\alpha+H>1$.

Combining (\ref{eq:moment})--(\ref{eq:It}) and (\ref{eq:IIt3})--(\ref{eq:IIIt}),
we obtain for $H\in(0,\frac12)$ that
\begin{align}\label{eq:EeH012}
\mathbb{E}\left|\int_0^t(t-\tau)^{\alpha-1}E_{\alpha,\alpha}(-\lambda_k(t-\tau)^
{\alpha})dB^H(\tau)\right|^2\lesssim t^{2\alpha+2H-2}.
\end{align}

\subsection{The case $H\in(\frac{1}{2},1)$}\label{subsec3.2}

It follows from Appendix again that the stochastic integral in (\ref{exIk2})
with respect to $B^H$ satisfies
\begin{align*}
&\mathbb{E}\left[\left(\int_0^t(t-\tau)^{\alpha-1}E_{\alpha,\alpha}
(-\lambda_k(t-\tau)^{\alpha})dB^H(\tau)\right)^2\right]\\
=&
\alpha_H\int_0^t\int_0^t(t-p)^{\alpha-1}E_{\alpha,\alpha}(-\lambda_k(t-p)^{
\alpha})(t-q)^{\alpha-1}E_{\alpha,\alpha}(-\lambda_k(t-q)^{\alpha})|p-q|^{2H-2}
dpdq
\end{align*}
By Lemma \ref{MLinq}, we have
\begin{align*}
&\mathbb{E}\left[\left(\int_0^t(t-\tau)^{\alpha-1}E_{\alpha,\alpha}
(-\lambda_k(t-\tau)^{\alpha})dB^H(\tau)\right)^2\right]\\
\lesssim&\alpha_H\int_0^t\int_0^t(t-p)^{\alpha-1}(t-q)^{\alpha-1}|p-q|^{2H-2}
dpdq\\
=&\alpha_H\int_0^t\int_0^t(t-p)^{\alpha-1}(t-q)^{\alpha-1}|(t-q)-(t-p)|^{2H-2}
dpdq.
\end{align*}
Let $t-p=\tilde{p}, t-q=\tilde{q}$. A simple calculation gives
\begin{align*}
&\int_0^t\int_0^t(t-p)^{\alpha-1}(t-q)^{\alpha-1}|(t-q)-(t-p)|^{2H-2}
dpdq\\
=&\int_0^t\int_0^t\tilde{p}^{\alpha-1}\tilde{q}^{\alpha-1}|\tilde{q}-\tilde{p}|^
{2H-2}d\tilde{p}d\tilde{q}\\
=&\int_0^t\int_0^{\tilde{q}}\tilde{p}^{\alpha-1}\tilde{q}^{\alpha-1}|\tilde{q}
-\tilde{p}|^{2H-2}d\tilde{p}d\tilde{q}+\int_0^t\int_{\tilde{q}}^t\tilde{p}^{
\alpha-1}\tilde{q}^{\alpha-1}|\tilde{q}-\tilde{p}|^{2H-2}d\tilde{p}d\tilde{q}\\
=&2\int_0^t\int_{\tilde{q}}^t\tilde{p}^{\alpha-1}\tilde{q}^{\alpha-1}(\tilde{p}
-\tilde{q})^{2H-2}d\tilde{p}d\tilde{q}\\
=&2\int_0^t\int_{\tilde{q}}^t\tilde{p}^{\alpha+2H-3}\tilde{q}^{\alpha-1}(1-\frac
{\tilde{q}}{\tilde{p}})^{2H-2}d\tilde{p}d\tilde{q}.
\end{align*}
Since $|\frac{\tilde{q}}{\tilde{p}}|<1$, we have from the binomial expansion
that
\begin{align*}
&\int_0^t\int_{\tilde{q}}^t\tilde{p}^{\alpha+2H-3}\tilde{q}^{\alpha-1}(1-\frac{
\tilde{q}}{\tilde{p}})^{2H-2}d\tilde{p}d\tilde{q}\\
=&\int_0^t\int_{\tilde{q}}^t\tilde{p}^{\alpha+2H-3}\tilde{q}^{\alpha-1}
\left(\sum_{n=0}^{\infty}\left(\begin{array}{c}
2H-2\\
n
\end{array}\right)(-\frac{\tilde{q}}{\tilde{p}})^n\right)d\tilde{p}d\tilde{q}\\
=&\int_0^t\tilde{q}^{\alpha-1}\sum_{n=0}^{\infty}\left(\begin{array}{c}
2H-2\\
n
\end{array}\right)(-1)^n\tilde{q}^n\left(\int_{\tilde{q}}^t\tilde{p}^{
\alpha+2H-3-n}d\tilde{p}\right)d\tilde{q}.
\end{align*}

Note that when $n=0$, $\alpha+2H-3-n=-1$ is possible, but when $n\geq1$,
$\alpha+2H-3-n=-1$ is impossible. Therefore we discuss the above
integral in two cases.

Case I: $2H-2=-\alpha$. It follows from the straightforward calculations that
\begin{align*}
&\int_0^t\tilde{q}^{\alpha-1}\sum_{n=0}^{\infty}\left(\begin{array}{c}
2H-2\\
n
\end{array}\right)(-1)^n\tilde{q}^n\left(\int_{\tilde{q}}^t\tilde{p}^{
\alpha+2H-3-n}d\tilde{p}\right)d\tilde{q}\\
=&\int_0^t\tilde{q}^{\alpha-1}\sum_{n=0}^{\infty}\left(\begin{array}{c}
2H-2\\
n
\end{array}\right)(-1)^n\tilde{q}^n\left(\int_{\tilde{q}}^t\tilde{p}^{-1-n}
d\tilde{p}\right)d\tilde{q}\\
=&\int_0^t\left[\tilde{q}^{\alpha-1}\int_{\tilde{q}}^t\tilde{p}^{-1}d\tilde{p}
+\tilde{q}^{\alpha-1}\sum_{n=1}^{\infty}\left(\begin{array}{c}
2H-2\\
n
\end{array}\right)(-1)^n\tilde{q}^n\left(\int_{\tilde{q}}^t\tilde{p}^{-1-n}
d\tilde{p}\right)\right]d\tilde{q}\\
=&\int_0^t\tilde{q}^{\alpha-1}(\ln t-\ln
\tilde{q})d\tilde{q}+\int_0^t\tilde{q}^{\alpha-1}\sum_{n=1}^{\infty}\left(\begin
{array}{c}
2H-2\\
n
\end{array}\right)(-1)^n\tilde{q}^n\frac{\tilde{q}^{-n}-t^{-n}}{n}d\tilde{q}\\
=&\ln
t\int_0^t\tilde{q}^{\alpha-1}d\tilde{q}-\lim_{\epsilon\rightarrow0+}\int_{
\epsilon}^t\tilde{q}^{\alpha-1}\ln\tilde{q}d\tilde{q}+\sum_{n=1}^{\infty}
\left(\begin{array}{c}
2H-2\\
n
\end{array}\right)\frac{(-1)^n}{n}\left(\int_0^t\tilde{q}^{\alpha-1}d\tilde{q}
-t^{-n}\int_0^t\tilde{q}^{\alpha+n-1}d\tilde{q}\right)\\
=&\frac{t^{\alpha}}{\alpha}\ln
t-\lim_{\epsilon\rightarrow0+}\left(\frac{t^{\alpha}}{\alpha}\ln
t-\frac{\epsilon^{\alpha}}{\alpha}\ln
\epsilon-\frac{t^{\alpha}-\epsilon^{\alpha}}{\alpha^2}\right)+\sum_{n=1}^{\infty
}\left(\begin{array}{c}
2H-2\\
n
\end{array}\right)\frac{(-1)^n}{n}\frac{nt^{\alpha}}{\alpha(\alpha+n)}\\
=&\frac{t^{\alpha}}{\alpha}\frac{1}{\alpha}+\frac{t^{\alpha}}{\alpha}\sum_{n=1}^
{\infty}\left(\begin{array}{c}
2H-2\\
n
\end{array}\right)(-1)^n\frac{1}{\alpha+n}\\
=&\frac{t^{\alpha}}{\alpha}\sum_{n=0}^{\infty}\left(\begin{array}{c}
2H-2\\
n
\end{array}\right)(-1)^n\frac{1}{\alpha+n},
\end{align*}
where the integration by parts and L'H\^{o}pital's rule are used. Moreover,
the condition $\alpha>0$ can guarantee the convergence of the singular
integrals.

Case II: $2H-2\neq-\alpha$. Similarly, we have from straightforward calculations
that
\begin{align*}
&\int_0^t\tilde{q}^{\alpha-1}\sum_{n=0}^{\infty}\left(\begin{array}{c}
2H-2\\
n
\end{array}\right)(-1)^n\tilde{q}^n\left(\int_{\tilde{q}}^t\tilde{p}^{
\alpha+2H-3-n}d\tilde{p}\right)d\tilde{q}\\
=&\sum_{n=0}^{\infty}\left(\begin{array}{c}
2H-2\\
n
\end{array}\right)(-1)^n\frac{1}{\alpha+2H-2-n}\left(t^{\alpha+2H-2-n}
\int_0^t\tilde{q}^{n+\alpha-1}d\tilde{q}-\int_0^t\tilde{q}^{2\alpha+2H-3}d\tilde
{q}\right),
\end{align*}
where the conditions $\alpha>0$ and $\alpha+H>1$ are needed to ensure the
convergence of the singular integrals. Then we have
\begin{align*}
&\sum_{n=0}^{\infty}\left(\begin{array}{c}
2H-2\\
n
\end{array}\right)(-1)^n\frac{1}{\alpha+2H-2-n}\left(t^{\alpha+2H-2-n}
\int_0^t\tilde{q}^{n+\alpha-1}d\tilde{q}-\int_0^t\tilde{q}^{2\alpha+2H-3}d\tilde
{q}\right).\\
=&\sum_{n=0}^{\infty}\left(\begin{array}{c}
2H-2\\
n
\end{array}\right)(-1)^n\frac{1}{\alpha+2H-2-n}\left(t^{\alpha+2H-2-n}\frac{t^{
n+\alpha}}{n+\alpha}-\frac{t^{2\alpha+2H-2}}{2\alpha+2H-2}\right)\\
=&\frac{t^{2\alpha+2H-2}}{2\alpha+2H-2}\sum_{n=0}^{\infty}\left(\begin{array}{c}
2H-2\\
n
\end{array}\right)(-1)^n\frac{1}{\alpha+n}.
\end{align*}

Combining Case I and Case II, we get
\begin{align*}
&\int_0^t\tilde{q}^{\alpha-1}\sum_{n=0}^{\infty}\left(\begin{array}{c}
2H-2\\
n
\end{array}\right)(-1)^n\tilde{q}^n\left(\int_{\tilde{q}}^t\tilde{p}^{
\alpha+2H-3-n}d\tilde{p}\right)d\tilde{q}\\
=&\frac{t^{2\alpha+2H-2}}{2\alpha+2H-2}\sum_{n=0}^{\infty}\left(\begin{array}{c}
2H-2\\
n
\end{array}\right)(-1)^n\frac{1}{\alpha+n}.
\end{align*}
It is easy to know from the asymptotic expansion for the binomial coefficients
that the above series is convergent. Therefore,
\begin{align}\label{eq:EeH121}
&\mathbb{E}\left[\left(\int_0^t(t-\tau)^{\alpha-1}E_{\alpha,\alpha}
(-\lambda_k(t-\tau)^{\alpha})dB^H(\tau)\right)^2\right]\lesssim
t^{2\alpha+2H-2}.
\end{align}

\subsection{Estimates of the solution}

In this section, we discuss the stability of the solution. From
(\ref{eq:EeH012})--(\ref{eq:EeH121}) and the analysis for $H=\frac{1}{2}$ in
\cite{Niu+2018}, for $0<H<1, 0<\alpha\leq1$ and $\alpha+H>1$, there holds
\begin{align}\label{eq:EeH01}
&\mathbb{E}\left[\left(\int_0^t(t-\tau)^{\alpha-1}E_{\alpha,\alpha}
(-\lambda_k(t-\tau)^{\alpha})dB^H(\tau)\right)^2\right]\lesssim
t^{2\alpha+2H-2}.
\end{align}
With the help of \eqref{eq:EeH01}, we obtain the stability estimates for the
mild solution (\ref{Dsolution}).

\begin{theorem}\label{thereom1}
Let $0<H<1, 0<\alpha\leq1$ and $\alpha+H>1$. Then the stochastic process $u$
given in (\ref{Dsolution}) satisfies
\begin{equation*}
\mathbb{E}(\|u\|^2_{L^2({D}\times[0,T])})\lesssim\|h\|^{2}_{L^2(0,T)}\|f\|^
{2}_{L^2({D})}+T^{2\alpha+2H-1}\|g\|^{2}_{L^2({D})}.
\end{equation*}
\end{theorem}

\begin{proof}
The proof follows easily from (\ref{ex}), (\ref{S1}), (\ref{exIk2}),
(\ref{eq:EeH01}). Especially, one can check it is also true for $\alpha=1$.
\end{proof}

Similarly, we also have the following stability results.
\begin{theorem}\label{thereom2}
Let $0<H<1, 0<\alpha\leq1$ and $\alpha+H>1$. The supremum of the expected norm
of the solution satisfies
\begin{equation*}
\sup_{0\leqslant t\leqslant
T}\mathbb{E}\left[\|u(\cdot,t)\|^2_{L^2({D})}\right]\lesssim\|h\|^{2}_{L^{
\infty}(0,T)}\|f\|^{2}_{L^2({D})}+T^{2\alpha+2H-2}\|g\|^{2}_{L^2({D})}
.
\end{equation*}
Moreover, if condition $g\in H^2({D})$ is added, there also holds
\begin{equation*}
\sup_{0\leqslant t\leqslant
T}\mathbb{E}\left[\|u(\cdot,t)\|^2_{H^2({D})}\right]\lesssim\|h\|^{2}_{L^{
\infty}(0,T)}\|f\|^{2}_{L^2({D})}+T^{2\alpha+2H-2}\|g\|^{2}_{H^2({D})}
.
\end{equation*}
\end{theorem}

\begin{proof}
The theorem can be proved by following similar arguments for the case
$H=\frac12, \alpha\in(\frac12,1)$ in \cite[Lemma 3.5]{Niu+2018}. The details
are omitted for brevity.
\end{proof}

Although we only show the details for the Laplacian operator in \eqref{fde},
the method can be applied to following initial-bound value problem for the
stochastic fractional diffusion equation with the fractional Laplacian operator:
\begin{equation*}
\left\{
\begin{array}{ll}
\partial_{t}^{\alpha}u(x,t)+(-\Delta)^{s}u(x,t)=f(x)h(t)+g(x)\dot{B}^H(t), &
(x,t)\in{D}\times(0,T),\\
u(x,t)=0, & (x,t)\in\mathbb{R}^n\backslash{D}\times[0,T], \\
u(x,0)=0, & x\in{D},
\end{array}
\right.
\end{equation*}
where $0<\alpha\leq1$, $0<s<1$, $0<H<1$, and $\alpha+H>1$. The fractional
Laplacian operator $(-\Delta)^{s}$ is defined as follows \cite[formula
(3.1)]{Nezza+2012}:
\begin{equation*}
(-\Delta)^{s}u(x,t)=C_{n,s}{\rm
p.v.}\int_{\mathbb{R}^n}\frac{u(x,t)-u(y,t)}{|x-y|^{ n+2s}}dy,
\end{equation*}
where $C_{n, s}$ is a positive constant depending on $n$ and $s$. Using the
properties of the eigensystem for the fractional
Laplacian operator $(-\Delta)^{s}$ given in \cite[Proposition
2.1]{Xavier+2016},
one can also use the method of separation of variables to obtain a mild
solution. Then all the results are the same except the second result in Theorem
\ref{thereom2}. But it can be easily shown that if $g\in H^s({D})$, then there
holds
\begin{equation*}
\sup_{0\leqslant t\leqslant
T}\mathbb{E}\left[\|u(\cdot,t)\|^2_{H^s({D})}\right]\lesssim\|h\|^{2}_{L^{
\infty}(0,T)}\|f\|^{2}_{L^2({D})}+T^{2\alpha+2H-2}\|g\|^{2}_{H^s({D})}
.
\end{equation*}
The fractional Sobolev space $H^s({D})$ can be found in \cite{Nezza+2012}
and related references therein.

\section{The inverse problem}

In this section, we consider the inverse problem of reconstructing $f$ and
$|g|$ from the empirical expectation and correlations of the final time data
$u(x,T)$. More specifically, the data may be assumed to be given by
\[
 u_k(T,\omega)=(u(T,\cdot,\omega),\varphi_k(\cdot))_{L^2(D)}.
\]
We shall discuss the uniqueness and the issue of instability,
separately.

It follows from (\ref{Dsolution})--(\ref{Ik2}) that we have
\begin{equation}\label{Iexp}
\mathbb{E}(u_k(T,\omega))=f_k\int_0^T(T-\tau)^{\alpha-1}E_{\alpha,\alpha}
(-\lambda_k(T-\tau)^{\alpha})h(\tau)d\tau
\end{equation}
and
\begin{equation}\label{Ivar}
\mathbb{V}(u_k(T,\omega))=g_k^2\, \mathbb{E}\left|\int_0^T(T-\tau)^{\alpha-1}E_{\alpha,\alpha}(-\lambda_k(T-\tau)^{\alpha})dB^H(\tau)\right|^2.
\end{equation}
Moreover, a straightforward calculation yields that
\begin{align}\label{Icvar}
&\text{Cov}(u_k(T,\omega),u_l(T,\omega))\notag\\
=&g_kg_l\,
\mathbb{E}\bigg[(\int_0^T(T-\tau)^{\alpha-1}E_{\alpha,\alpha}(-\lambda_k(T-\tau)
^{\alpha})dB^H(\tau))\notag\\
&\quad \times(\int_0^T(T-\tau)^{\alpha-1}E_{\alpha,\alpha}(-\lambda_l(
T -\tau)^{\alpha})dB^H(\tau))\bigg].
\end{align}

\begin{lemma}\label{Lowbnd1}
Suppose Assumption \ref{assumption} holds. For each fixed $k\in \mathbb{N}$, there exists a constant $C_1>0$ such that
\begin{equation*}
\int_0^T(T-\tau)^{\alpha-1}E_{\alpha,\alpha}(-\lambda_k(T-\tau)^{\alpha})h(\tau)d\tau\geq C_1>0.
\end{equation*}
\end{lemma}

\begin{proof}
Letting $\tilde{\tau}=T-\tau$, we have from Lemma \ref{Epositive} and Assumption
\ref{assumption} that 
\begin{align*}
&\int_0^T(T-\tau)^{\alpha-1}E_{\alpha,\alpha}(-\lambda_k(T-\tau)^{\alpha}
)h(\tau)d\tau\\
=&\int_0^T\tilde{\tau}^{\alpha-1}E_{\alpha,\alpha}(-\lambda_k\tilde{\tau}^{
\alpha})h(T-\tilde{\tau})d\tilde{\tau}\\
\geq & T^{\alpha-1}E_{\alpha,\alpha}(-\lambda_kT^{\alpha})\int_0^T
h(T-\tilde{\tau})d\tilde{\tau}\\
\geq & c_h T^{\alpha}E_{\alpha,\alpha}(-\lambda_kT^{\alpha})=:C_1>0,
\end{align*}
which completes the proof. 
\end{proof}

\begin{lemma}\label{Lowbnd2}
Suppose Assumption \ref{assumption} holds. For each fixed $k,l\in \mathbb{N}$, there exists a constant $C_2>0$ such that
\begin{align*}
&\mathbb{E}\bigg[(\int_0^T(T-\tau)^{\alpha-1}E_{\alpha,\alpha}(-\lambda_k(T-\tau)
^{\alpha})dB^H(\tau))\\
&\quad \times(\int_0^T(T-\tau)^{\alpha-1}E_{\alpha,\alpha}(-\lambda_l(
T -\tau)^{\alpha})dB^H(\tau))\bigg]\geq C_2>0.
\end{align*}
\end{lemma}

\begin{proof}
Let 
$\phi_k(s)=(T-s)^{\alpha-1}E_{\alpha,\alpha}(-\lambda_k(T-s)^{\alpha }
)$ and 
\[
I_{kl}=\mathbb{E}\bigg[
\big(\int_0^T\phi_k(\tau)dB^H(\tau)\big)\big(\int_0^T\phi_l(\tau)dB^H(\tau)\big)
\bigg] .
\]
We consider $I_{kl}$ for $H\in(\frac12,1)$ and $H\in(0,\frac12)$, separately.

For $H\in(\frac12,1)$, we have from \eqref{V121} that 
\begin{align*}
I_{kl}&=\alpha_H\int_0^T\int_0^T\phi_k(r)\phi_l(u)|r-u|^{2H-2}dudr\\
&=\alpha_H\int_0^T\int_0^T(T-r)^{\alpha-1}E_{\alpha,\alpha}(-\lambda_k(T-r)^{
\alpha})(T-u)^{\alpha-1}E_{\alpha,\alpha}(-\lambda_l(T-u)^{\alpha})|r-u|^{2H-2}
dudr.
\end{align*}
Set $\tilde{r}=T-r, \tilde{u}=T-u$. A simple calculation yields 
\begin{align*}
I_{kl}&=\alpha_H\int_0^T\int_0^T\tilde{r}^{\alpha-1}E_{\alpha,\alpha}
(-\lambda_k\tilde{r}^{\alpha})\tilde{u}^{\alpha-1}E_{\alpha,\alpha}
(-\lambda_l\tilde{u}^{\alpha})|\tilde{u}-\tilde{r}|^{2H-2}d\tilde{u}d\tilde{r}.
\end{align*}
By Lemma \ref{Epositive}, the
function
$\tilde{\tau}^{\alpha-1}E_{\alpha,\alpha}(-\lambda_k\tilde{\tau}^{\alpha })$ is
a monotonically decreasing function with respect to $\tilde\tau$. Hence
\begin{align*}
I_{kl}&\geq\alpha_H\int_0^T\int_0^TT^{\alpha-1}E_{\alpha,\alpha}(-\lambda_kT^{\alpha})T^{\alpha-1}E_{\alpha,\alpha}(-\lambda_lT^{\alpha})|\tilde{u}-\tilde{r}|^{2H-2}d\tilde{u}d\tilde{r}\\
&=\alpha_H
T^{2\alpha-2}E_{\alpha,\alpha}(-\lambda_kT^{\alpha})E_{\alpha,\alpha}
(-\lambda_lT^{\alpha})\int_0^T\int_0^T|\tilde{u}-\tilde{r}|^{2H-2}d\tilde{u}
d\tilde{r}\\
&=\frac{\alpha_H
T^{2(\alpha+H-1)}}{H(2H-1)}E_{\alpha,\alpha}(-\lambda_kT^{
\alpha})E_{ \alpha,\alpha}(-\lambda_lT^{\alpha})=:C_2>0.
\end{align*}

For $H\in(0,\frac12)$, by \eqref{V012}, we have
\begin{align*}
I_{kl}&=\langle K_{H,T}^*\phi_k,K_{H,T}^*\phi_l\rangle_{L^2(0,T)},
\end{align*}
where 
\[
(K_{H,T}^*\phi_k)(s)=K_H(T,
s)\phi_k(s)+\int_s^T(\phi_k(u)-\phi_k(s))\frac{\partial
K_H(u,s)}{\partial u}du,
\]
\[
K_H(T,s)=c_H\left[\left(\frac Ts\right)^{H-\frac12}(T-s)^{H-\frac12}-\left(H-\frac12\right)s^{\frac12-H}\int_s^Tu^{H-\frac32}(u-s)^{H-\frac12}du\right],
\]
\[
\frac{\partial K_H(u,s)}{\partial u}=c_H\left(\frac us\right)^{H-\frac12}(u-s)^{H-\frac32}.
\]
Obviously, $K_H(T,s)>0$ since $H\in(0,\frac12), 0<s<T$. It follows from the
mean value theorem that 
\begin{align*}
&\int_s^T(\phi_k(u)-\phi_k(s))\frac{\partial K_H(u,s)}{\partial u}du\\
=&c_H \int_s^T(\phi_k(u)-\phi_k(s))\left(\frac
us\right)^{H-\frac12}(u-s)^{H-\frac32}du\\
=&c_H \int_s^T\phi'_k(u_k^*)\left(\frac
us\right)^{H-\frac12}(u-s)^{H-\frac12}du\qquad(s<u_k^*<u<T)\\
=&c_H\phi'_k(u_k^{**})\int_s^T\left(\frac
us\right)^{H-\frac12}(u-s)^{H-\frac12}du\qquad(s<u_k^{**}<T)\\
=&M_H(s)\phi'_k(u_k^{**}),
\end{align*}
where 
\[
M_H(s)= c_H\int_s^T\left(\frac
us\right)^{H-\frac12}(u-s)^{H-\frac12}du>0.
\]
A simple calculation gives that 
\begin{align*}
I_{kl}=&\langle K_H(T,s)\phi_k(s)+
M_H(s)\phi'_k(u_k^{**}),K_H(T,s)\varphi_l(s)+
M_H(s)\phi'_l(u_l^{**})\rangle_{L^2(0,T)}\\
=&\langle K_H(T,s)\phi_k(s),K_H(T,s)\phi_l(s)\rangle_{L^2(0,T)}+\langle
M_H(s)\phi'_k(u_k^{**}),M_H(s)\phi'_l(u_l^{**})\rangle_{L^2(0,T)}\\
&+\langle
K_H(T,s)\phi_k(s),M_H(s)\phi'_l(u_l^{**})\rangle_{L^2(0,T)}+\langle
M_H(s)\phi'_k(u_k^{**}),K_H(T,s)\phi_l(s)\rangle_{L^2(0,T)}.
\end{align*}
It follows Lemma \ref{Epositive} again that there holds 
\begin{align*}
&\langle K_H(T,s)\phi_k(s),K_H(T,s)\phi_l(s)\rangle_{L^2(0,T)}\\
=&\int_0^TK_H^2(T,s)\phi_k(s)\phi_l(s)ds\\
\geq&T^{2\alpha-2}E_{\alpha,\alpha}(-\lambda_k
T^\alpha)E_{\alpha,\alpha}(-\lambda_l
T^\alpha)\int_0^TK_H^2(T,s)ds:=\tilde c_1>0.
\end{align*}
Using Lemmas \ref{dEalf} and \ref{c.m.}, and noting $T-s=t$, we obtain that 
$\phi_k(s)> 0$ and $\phi_k(s)$ is a monotonically increasing function;
$\phi_k'(s)> 0$ and $\phi_k'(s)$ is a monotonically decreasing function,
which imply $\phi_k(s)\geq \phi_k(0)>0$ and $\varphi'_k(s)\geq\varphi'_k(T)>0$
for $0<s<T$. Hence
\[
\langle
M_H(s)\phi'_k(u_k^{**}),M_H(s)\phi'_l(u_l^{**})\rangle_{L^2(0,T)}
\geq\phi'_k(T)\phi'_l(T)\int_0^TM_H^2(s)ds:=\tilde c_2>0.
\] 
Similarly, we have
\[
\langle
K_H(T,s)\phi_k(s),M_H(s)\phi'_l(u_l^{**})\rangle_{L^2(0,T)}
\geq\phi_k(0)\phi'_l(T)\int_0^TK_H(T,s)M_H(s)ds:=\tilde c_3>0
\]
and
\[
\langle
M_H(s)\phi'_k(u_k^{**}),K_H(T,s)\phi_l(s)\rangle_{L^2(0,T)}
\geq\phi'_k(T)\phi_l(0)\int_0^TK_H(T,s)M_H(s)ds:=\tilde c_4>0.
\]
Combining the above estimates gives
\[
 I_{kl}\geq \sum_{j=1}^4\tilde c_j:=C_2>0,
\]
which completes the proof.
\end{proof}

Combining \eqref{Iexp}--\eqref{Icvar} and Lemmas \ref{Lowbnd1} and \ref{Lowbnd2}, we obtain the uniqueness of the inverse problem.

\begin{theorem}
Suppose Assumption \ref{assumption} holds. Then the quantities
\begin{equation*}
\{\mathbb{E}(u_k(T,\omega)), \text{Cov}(u_k(T,\omega),u_l(T,\omega)): k,l \in
\mathbb{N}\}
\end{equation*}
can determine the source terms $f$ and $|g|$ uniquely.
\end{theorem}

\begin{proof}
 Since $f, g\in L^2(D)$, we have
 \[
  f(x)=\sum_{k=1}^\infty f_k \varphi_k(x),\quad  g(x)=\sum_{k=1}^\infty g_k
\varphi_k(x),
 \]
which gives that
\[
 g^2(x)=\left( \sum_{k=1}^\infty g_k
\varphi_k(x)\right) \left( \sum_{l=1}^\infty g_l
\varphi_l(x)\right)=\sum_{k,l\in\mathbb N} g_k g_l\varphi_k(x)\varphi_l(x).
\]
By Lemmas \ref{Lowbnd1} and \ref{Lowbnd2}, the proof is completed by noting \eqref{Iexp} and \eqref{Icvar}.
\end{proof}

Unfortunately, the inverse source problem is unstable. In \cite[Lemma
4.4]{Niu+2018}, the authors have explained the instability for $H=\frac12,
\alpha\in(\frac12,1)$. Since the formula (\ref{Iexp}) does not involve the
Brownian motion, the instability of recovering $f$ is the same. Therefore, we
shall only discuss the instability of recovering $|g|$. It suffices to show
that it is unstable to recover $g_k^2$ in \eqref{Icvar} when $k=l$, i.e., we
shall focus on the estimate of \eqref{Ivar}.

First, we choose $t_*$ small enough such that
\begin{align}\label{t*}
\frac1{1+\lambda_kr^\alpha}\le
\begin{cases}
1 &\quad\text{if} ~ r<t_*,\\
\frac1{\lambda_kr^{\alpha}} &\quad\text{if} ~ r>t_*.
\end{cases}
\end{align}
Below we discuss the two different cases $H\in(0,\frac{1}{2})$ and
$H\in(\frac{1}{2}, 1)$, separately.

\subsection{The case $H\in(0,\frac{1}{2})$}

We consider the estimate \eqref{eq:moment} with $t=T$ and estimate
$I_j(T), j=1, 2, 3$.

The estimate $I_1(T)$. A simple calculation yields
\begin{align*}
I_1(T)=&\int_0^T\left[\left(\frac
{T}{\tau}\right)^{H-\frac12}(T-\tau)^{\alpha+H-\frac32}E_{\alpha,\alpha}(-
\lambda_k(T-\tau)^{\alpha})\right]^2d\tau\\
=&\left(\int_0^{T-t_*}+\int_{T-t_*}^T\right)\left(\left(\frac {T}{\tau}\right)^{2H-1}(T-\tau)^{2\alpha+2H-3}E_{\alpha,\alpha}^2(-\lambda_k(T-\tau)^{\alpha})\right)d\tau.
\end{align*}
We have from (\ref{t*}) that
\begin{align*}
&\int_0^{T-t_*}\left(\frac {T}{\tau}\right)^{2H-1}(T-\tau)^{2\alpha+2H-3}E_{\alpha,\alpha}^2(-\lambda_k(T-\tau)^{\alpha})\,d\tau\\
\leq&\int_0^{T-t_*}\left(\frac {T}{\tau}\right)^{2H-1}(T-\tau)^{2\alpha+2H-3}\frac{1}{\lambda_k^2(T-\tau)^{2\alpha}}\,d\tau\\
=&\frac{1}{\lambda_k^2}\int_0^{T-t_*}\left(\frac {T}{\tau}\right)^{2H-1}(T-\tau)^{2H-3}\,d\tau.
\end{align*}
The condition $H>0$ is enough to ensure the convergence of the above singular
integral. Moreover, it follows from the binomial expansion that we obtain
\begin{align*}
&\frac{1}{\lambda_k^2}\int_0^{T-t_*}\left(\frac {T}{\tau}\right)^{2H-1}(T-\tau)^{2H-3}\,d\tau\\
=&\frac{T^{4H-4}}{\lambda_k^2}\int_0^{T-t_*}\tau^{1-2H}(1-\frac{\tau}{T})^{2H-3}\,d\tau\\
=&\frac{T^{4H-4}}{\lambda_k^2}\int_0^{T-t_*}\tau^{1-2H}\left(\sum_{n=0}^{\infty}\left(\begin{array}{c}
2H-3\\
n
\end{array}\right)(-1)^n\left(\frac{1}{T}\right)^n\tau^n\right)d\tau\\
=&\frac{T^{4H-4}}{\lambda_k^2}\sum_{n=0}^{\infty}\left(\begin{array}{c}
2H-3\\
n
\end{array}\right)(-1)^n\left(\frac{1}{T}\right)^n\int_0^{T-t_*}\tau^{n+1-2H}d\tau\\
=&\frac{T^{4H-4}}{\lambda_k^2}\sum_{n=0}^{\infty}\left(\begin{array}{c}
2H-3\\
n
\end{array}\right)(-1)^n\left(\frac{1}{T}\right)^n\frac{(T-t_*)^{n+2-2H}}{n+2-2H}\\
=&\frac{T^{4H-4}}{\lambda_k^2}(T-t_*)^{2-2H}\sum_{n=0}^{\infty}\left(\begin{array}{c}
2H-3\\
n
\end{array}\right)\frac{(-1)^n}{n+2-2H}\left(\frac{T-t_*}{T}\right)^n\\
\lesssim&\frac{1}{\lambda_k^2}T^{4H-4}(T-t_*)^{2-2H}
\leq\frac{1}{\lambda_k^2}T^{2H-2}.
\end{align*}
On the other hand, by (\ref{t*}), there holds
\begin{align*}
&\int_{T-t_*}^T\left(\frac {T}{\tau}\right)^{2H-1}(T-\tau)^{2\alpha+2H-3}E_{\alpha,\alpha}^2(-\lambda_k(T-\tau)^{\alpha})d\tau.\\
\lesssim&\int_{T-t_*}^T\left(\frac {T}{\tau}\right)^{2H-1}(T-\tau)^{2\alpha+2H-3}d\tau.
\end{align*}
Clearly, this singular integral needs the condition $\alpha+H>1$ to guarantee
the convergence. By the mean value theorem for the definite integral,
there exists $T-t_*<\xi<T$ such that
\begin{align*}
&\int_{T-t_*}^T\left(\frac {T}{\tau}\right)^{2H-1}(T-\tau)^{2\alpha+2H-3}d\tau\\
=&\left(\frac{T}{\xi}\right)^{2H-1}\int_{T-t_*}^T(T-\tau)^{2\alpha+2H-3}d\tau\\
=&\left(\frac{T}{\xi}\right)^{2H-1}\frac{t_*^{2\alpha+2H-2}}{2\alpha+2H-2}
\lesssim t_*^{2\alpha+2H-2}.
\end{align*}
Combining the above estimate leads to
\begin{equation}\label{ITe}
I_1(T)\lesssim\frac{1}{\lambda_k^2}T^{2H-2}+t_*^{2\alpha+2H-2}.
\end{equation}

The estimate of $I_2(T)$. Using Lemma \ref{MLinq} and (\ref{t*}), we
have
\begin{align}\label{IIT}
I_2(T)=&\int_0^T{\tau}^{1-2H}\left[\left(\int_\tau^Tu^{H-\frac32}(u-\tau)^{H-
\frac12}du\right)(T-\tau)^{\alpha-1}E_{\alpha,\alpha}(-\lambda_k(T-\tau)^{\alpha
}
)\right]^2d\tau\nonumber\\
\lesssim&\int_0^T{\tau}^{1-2H}\left(\int_\tau^Tu^{H-\frac32}(u-\tau)^{H-\frac12}du\right)^2(T-\tau)^{2\alpha-2}\left(\frac{1}{1+\lambda_k(T-\tau)^{\alpha}}\right)^2d\tau\nonumber\\
=&\int_0^{T-t_*}{\tau}^{1-2H}\left(\int_\tau^{T-t_*}u^{H-\frac32}(u-\tau)^{H-\frac12}du\right)^2(T-\tau)^{2\alpha-2}\left(\frac{1}{1+\lambda_k(T-\tau)^{\alpha}}\right)^2d\tau\nonumber\\
&+\int_0^{T-t_*}{\tau}^{1-2H}\left(\int_{T-t_*}^Tu^{H-\frac32}(u-\tau)^{H-\frac12}du\right)^2(T-\tau)^{2\alpha-2}\left(\frac{1}{1+\lambda_k(T-\tau)^{\alpha}}\right)^2d\tau\nonumber\\
&+\int_{T-t_*}^T{\tau}^{1-2H}\left(\int_{\tau}^Tu^{H-\frac32}(u-\tau)^{H-\frac12}du\right)^2(T-\tau)^{2\alpha-2}\left(\frac{1}{1+\lambda_k(T-\tau)^{\alpha}}\right)^2d\tau\nonumber\\
\lesssim&\int_0^{T-t_*}{\tau}^{1-2H}\left(\int_\tau^{T-t_*}u^{H-\frac32}(u-\tau)^{H-\frac12}du\right)^2(T-\tau)^{2\alpha-2}\frac{1}{\lambda_k^2(T-\tau)^{2\alpha}}d\tau\nonumber\\
&+\int_0^{T-t_*}{\tau}^{1-2H}\left(\int_{T-t_*}^Tu^{H-\frac32}(u-\tau)^{H-\frac12}du\right)^2(T-\tau)^{2\alpha-2}\frac{1}{\lambda_k^2(T-\tau)^{2\alpha}}d\tau\nonumber\\
&+\int_{T-t_*}^T{\tau}^{1-2H}\left(\int_{\tau}^Tu^{H-\frac32}(u-\tau)^{H-\frac12}du\right)^2(T-\tau)^{2\alpha-2}d\tau\nonumber\\
=:&J_1(T)+J_2(T)+J_3(T).
\end{align}
Next we estimate $J_j(T), j=1, 2, 3$, respectively.

For the term $J_1(T)$, we get
\begin{align}\label{IIT11}
J_1(T)=&\frac{1}{\lambda_k^2}\int_0^{T-t_*}{\tau}^{1-2H}(T-\tau)^{-2}
\left(\int_\tau^{T-t_*}u^{H-\frac32}(u-\tau)^{H-\frac12}
du\right)^2d\tau\nonumber\\
=&\frac{1}{\lambda_k^2}\int_0^{T-t_*}{\tau}^{1-2H}(T-\tau)^{-2}\left[\int_\tau^{T-t_*}u^{2H-2}\sum_{n=0}^{\infty}\left(
\begin{array}{c}
H-\frac12\\
n
\end{array}\right)(-1)^nu^{-n}\tau^ndu\right]^2d\tau\nonumber\\
=&\frac{1}{\lambda_k^2}\int_0^{T-t_*}{\tau}^{1-2H}(T-\tau)^{-2}\left[\sum_{n=0}^{\infty}\left(
\begin{array}{c}
H-\frac12\\
n
\end{array}\right)(-1)^n\tau^n\int_\tau^{T-t_*}u^{2H-2-n}du\right]^2d\tau\nonumber\\
=&\frac{1}{\lambda_k^2}\int_0^{T-t_*}{\tau}^{1-2H}(T-\tau)^{-2}\bigg[(T-t_*)^{
2H-1 } \sum_{n=0}^{\infty}\left(
\begin{array}{c}
H-\frac12\\
n
\end{array}\right)\frac{(-1)^n}{2H-1-n}\left(\frac{\tau}{T-t_*}\right)^n\nonumber\\
&-\tau^{2H-1}\sum_{n=0}^{\infty}\left(
\begin{array}{c}
H-\frac12\\
n
\end{array}\right)\frac{(-1)^n}{2H-1-n}\bigg]^2d\tau\nonumber\\
\leq&\frac{1}{\lambda_k^2}\int_0^{T-t_*}{\tau}^{1-2H}(T-\tau)^{-2}\left[(T-t_*)^{4H-2}+\tau^{4H-2}\right]d\tau
\nonumber\\
=&\frac{1}{\lambda_k^2}\left((T-t_*)^{4H-2}\int_0^{T-t_*}{\tau}^{1-2H}(T-\tau)^{
-2}d\tau+\int_0^{T-t_*}{\tau}^{2H-1}(T-\tau)^{-2}d\tau\right),
\end{align}
where the condition $0<H<\frac12$ is used to make the above singular
integrals convergent. Hence
\begin{align}\label{IIT12}
J_1(T)\lesssim&\frac{1}{\lambda_k^2}\left((T-t_*)^{4H-2}\int_0^{T-t_*}{\tau}^{
1-2H}(T-\tau)^{-2}d\tau+\int_0^{T-t_*}{\tau}^{2H-1}(T-\tau)^{-2}
d\tau\right)\nonumber\\
=&\frac{1}{\lambda_k^2}\bigg[(T-t_*)^{4H-2}T^{-2}\int_0^{T-t_*}\tau^{1-2H}
\left(\sum_{n=0}^{\infty}\left(
\begin{array}{c}
-2\\
n
\end{array}\right)(-1)^nT^{-n}\tau^n\right)d\tau\nonumber\\
&+T^{-2}\int_0^{T-t_*}\tau^{2H-1}\left(\sum_{n=0}^{\infty}\left(
\begin{array}{c}
-2\\
n
\end{array}\right)(-1)^nT^{-n}\tau^n\right)d\tau\bigg]\nonumber\\
=&\frac{1}{\lambda_k^2}\bigg[T^{-2}(T-t_*)^{2H}\sum_{n=0}^{\infty}\left(
\begin{array}{c}
-2\\
n
\end{array}\right)\frac{(-1)^n}{2-2H+n}\left(\frac{T-t_*}{T}\right)^n\nonumber\\
&+T^{-2}(T-t_*)^{2H}\sum_{n=0}^{\infty}\left(
\begin{array}{c}
-2\\
n
\end{array}\right)\frac{(-1)^n}{2H+n}\left(\frac{T-t_*}{T}\right)^n\bigg]
\lesssim \frac{1}{\lambda_k^2}T^{2H-2}.
\end{align}

For the second term $J_2(T)$ in (\ref{IIT}), we have
\begin{align}\label{IIT2}
J_2(T)=&\frac{1}{\lambda_k^2}\int_0^{T-t_*}{\tau}^{1-2H}(T-\tau)^{-2}
\left(\int_{T-t_*}^Tu^{H-\frac32}(u-\tau)^{H-\frac12}du\right)^2d\tau\nonumber\\
=&\frac{1}{\lambda_k^2}\int_0^{T-t_*}{\tau}^{1-2H}(T-\tau)^{-2}\left[\sum_{n=0}^{\infty}\left(
\begin{array}{c}
H-\frac12\\
n
\end{array}\right)(-1)^n\tau^n\int_{T-t_*}^Tu^{2H-2-n}du\right]^2d\tau\nonumber\\
=&\frac{1}{\lambda_k^2}\int_0^{T-t_*}{\tau}^{1-2H}(T-\tau)^{-2}\bigg[T^{2H-1}
\sum_ { n=0}^{\infty}\left(
\begin{array}{c}
H-\frac12\\
n
\end{array}\right)\frac{(-1)^n}{2H-1-n}\left(\frac{\tau}{T}\right)^n\nonumber\\
&-(T-t_*)^{2H-1}\sum_{n=0}^{\infty}\left(
\begin{array}{c}
H-\frac12\\
n
\end{array}\right)\frac{(-1)^n}{2H-1-n}\left(\frac{\tau}{T-t_*}\right)^n\bigg]
^2d\tau\nonumber\\
\lesssim&\frac{1}{\lambda_k^2}\int_0^{T-t_*}{\tau}^{1-2H}(T-\tau)^{-2}\left[T^{4H-2}+(T-t_*)^{4H-2}\right]d\tau\nonumber\\
=&\frac{1}{\lambda_k^2}\left[T^{4H-2}+(T-t_*)^{4H-2}\right]T^{-2}(T-t_*)^{2-2H}\sum_{n=0}^{\infty}\left(
\begin{array}{c}
-2\\
n
\end{array}\right)\frac{(-1)^n}{2-2H+n}\left(\frac{T-t_*}{T}\right)^n\nonumber\\
\lesssim&\frac{1}{\lambda_k^2}T^{2H-2}.
\end{align}
In (\ref{IIT2}), we have used the same tricks as those in (\ref{IIT11}) and
(\ref{IIT12}).

For the third term $J_3(T)$, we obtain
\begin{align*}
J_3(T)=&\int_{T-t_*}^T{\tau}^{1-2H}(T-\tau)^{2\alpha-2}\left(\int_{\tau}^Tu^{
H-\frac32}(u-\tau)^{H-\frac12}du\right)^2d\tau\\
=&\int_{T-t_*}^T{\tau}^{1-2H}(T-\tau)^{2\alpha-2}\left(\sum_{n=0}^{\infty}\left(
\begin{array}{c}
H-\frac12\\
n
\end{array}\right)(-1)^n\tau^n\int_{\tau}^Tu^{2H-2-n}du\right)^2d\tau\\
=&\int_{T-t_*}^T{\tau}^{1-2H}(T-\tau)^{2\alpha-2}\bigg(T^{2H-1}\sum_{n=0}^{
\infty } \left(
\begin{array}{c}
H-\frac12\\
n
\end{array}\right)\frac{(-1)^n}{2H-1-n}\left(\frac{\tau}{T}\right)^n\\
&-\tau^{2H-1}\sum_{n=0}^{\infty}\left(
\begin{array}{c}
H-\frac12\\
n
\end{array}\right)\frac{(-1)^n}{2H-1-n}\bigg)^2d\tau\\
\lesssim&\int_{T-t_*}^T{\tau}^{1-2H}(T-\tau)^{2\alpha-2}\left(T^{4H-2}+\tau^{
4H-2}\right)d\tau\\
=&T^{4H+2\alpha-4}\int_{T-t_*}^T{\tau}^{1-2H}\left(\sum_{n=0}^{\infty}\left(
\begin{array}{c}
2\alpha-2\\
n
\end{array}\right)(-1)^nT^{-n}\tau^n\right)d\tau\\
&+T^{2\alpha-2}\int_{T-t_*}^T{\tau}^{2H-1}\left(\sum_{n=0}^{\infty}\left(
\begin{array}{c}
2\alpha-2\\
n
\end{array}\right)(-1)^nT^{-n}\tau^n\right)d\tau\\
=&T^{4H+2\alpha-4}\sum_{n=0}^{\infty}\left(
\begin{array}{c}
2\alpha-2\\
n
\end{array}\right)(-1)^nT^{-n}\frac{T^{2-2H+n}-(T-t_*)^{2-2H+n}}{2-2H+n}\\
&+T^{2\alpha-2}\sum_{n=0}^{\infty}\left(
\begin{array}{c}
2\alpha-2\\
n
\end{array}\right)(-1)^nT^{-n}\frac{T^{2H+n}-(T-t_*)^{2H+n}}{2H+n}.
\end{align*}
Noting the range of $H$, we can use the differential mean value theorem
and the H\"{o}lder continuity of $x^{2H}$ to obtain
\begin{align}\label{IIT32}
J_2(T)\lesssim T^{2H+2\alpha-3}t_*+T^{2\alpha-2}(t_*^{2H}\vee
t_*(T-t_*)^{2H-1}).
\end{align}
Combining (\ref{IIT}), (\ref{IIT12}), (\ref{IIT2}), and (\ref{IIT32}), we have
\begin{align}\label{IITe}
I_2(T)\lesssim \lambda_k^{-2}+(t_*^{2H}\vee t_*).
\end{align}

The estimate of $I_3(T)$. According to Lemma \ref{lm:tE},
\begin{align*}
I_3(T)\lesssim&\int_0^T\left[\int_\tau^T\left[\int_{T-u}^{T-\tau}\frac{r^{
\alpha-2 } }{1+\lambda_kr^{\alpha}}dr\right]\left(\frac
u\tau\right)^{H-\frac12}(u-\tau)^{H-\frac32}du\right]^2d\tau\\
\lesssim&\int_0^{T-t_*}\left[\int_\tau^{T-t_*}\left[\int_{T-u}^{T-\tau}\frac{r^{\alpha-2}}{1+\lambda_kr^{\alpha}}dr\right]\left(\frac u\tau\right)^{H-\frac12}(u-\tau)^{H-\frac32}du\right]^2d\tau\\
&+\int_0^{T-t_*}\left[\int_{T-t_*}^T\left[\int_{T-u}^{T-\tau}\frac{r^{\alpha-2}}{1+\lambda_kr^{\alpha}}dr\right]\left(\frac u\tau\right)^{H-\frac12}(u-\tau)^{H-\frac32}du\right]^2d\tau\\
&+\int_{T-t_*}^T\left[\int_\tau^T\left[\int_{T-u}^{T-\tau}\frac{r^{\alpha-2}}{1+\lambda_kr^{\alpha}}dr\right]\left(\frac u\tau\right)^{H-\frac12}(u-\tau)^{H-\frac32}du\right]^2d\tau\\
=&:K_1(T)+K_2(T)+K_3(T).
\end{align*}
Next is to estimate $K_j(T), j=1, 2, 3.$

For $K_1(T)$, a simple calculation gives
\begin{align*}
K_1(T)\lesssim&\int_0^{T-t_*}\left[\int_\tau^{T-t_*}\left[\frac1{\lambda_k}
\int_{T-u}^{T-\tau}r^{-2}dr\right]\left(\frac
u\tau\right)^{H-\frac12}(u-\tau)^{H-\frac32}du\right]^2d\tau\\
\lesssim&\frac1{\lambda_k^2}\int_0^{T-t_*}\left[\int_\tau^{T-t_*}(T-u)^{-2}(u-\tau)^{H-\frac12}\left(\frac u\tau\right)^{H-\frac12}du\right]^2d\tau.
\end{align*}
Noting $H\in(0,\frac12)$, we have
\begin{align*}
K_1(T)\lesssim&\frac1{\lambda_k^2}\int_0^{T-t_*}\left[\int_0^{T-t_*-\tau}
(T-\tau-r)^{-2}r^{H-\frac12}dr\right]^2d\tau\\
\lesssim&\frac1{\lambda_k^2}\int_0^{T-t_*}(T-\tau)^{-4}\left[\sum_{n=0}^{\infty}
\left(\begin{array}{c}
-2\\
n
\end{array}\right)(-1)^n(T-\tau)^{-n}\int_0^{T-t_*-\tau}r^{n+H-\frac12}dr\right]^2d\tau\\
=&\frac1{\lambda_k^2}\int_0^{T-t_*}(T-\tau)^{-4}(T-t_*-\tau)^{2H+1}d\tau\left[\sum_{n=0}^{\infty}
\left(\begin{array}{c}
-2\\
n
\end{array}\right)(-1)^n\frac{(T-t_*-\tau)^n}{(T-\tau)^n}\frac1{n+H+\frac12}\right]^2\\
\lesssim&\frac1{\lambda_k^2}\int_0^{T-t_*}(T-\tau)^{2H-3}d\tau
\lesssim\frac1{\lambda_k^2t_*^{2-2H}}.
\end{align*}

For $K_2(T)$, noting that $t_*<T-\tau<T$ and $0<T-u<t_*$ since $0<\tau<T-t_*$
and $T-t_*<u<T$, we get
\begin{align*}
\int_{T-u}^{T-\tau}\frac{r^{\alpha-2}}{1+\lambda_kr^{\alpha}}dr
\le&\int_{T-u}^{t_*}r^{\alpha-2}dr+\int_{t_*}^{T-\tau}\frac1{\lambda_k}r^{-2}dr\\
\le&\left(t_*^\alpha\vee\frac1{\lambda_k}\right)\int_{T-u}^{T-\tau}r^{-2}dr\\
\le&\left(t_*^\alpha\vee\frac1{\lambda_k}\right)(T-u)^{-2}(u-\tau).
\end{align*}
As a result, we have for $H\in(0,\frac12)$ that
\begin{align*}
K_2(T)\lesssim&\int_0^{T-t_*}\left[\int_{T-t_*}^T\left(t_*^\alpha\vee\frac1{
\lambda_k}\right)(T-u)^{-2}(u-\tau)^{H-\frac12}\left(\frac
u\tau\right)^{H-\frac12}du\right]^2d\tau\\\lesssim&\left(t_*^{2\alpha}\vee\frac1
{\lambda_k^2}\right)\int_0^{T-t_*}\left[\int_{T-t_*-\tau}^{T-\tau}(T-\tau-r)^{-2
}r^{H-\frac12}dr\right]^2d\tau\\
\lesssim&\left(t_*^{2\alpha}\vee\frac1{\lambda_k^2}\right)\int_0^{T-t_*}(T-\tau)^{-4}\left[\sum_{n=0}^{\infty}
\left(\begin{array}{c}
-2\\
n
\end{array}\right)(-1)^n(T-\tau)^{-n}\int_{T-t_*-\tau}^{T-\tau}r^{n+H-\frac12}dr\right]^2d\tau\\
\lesssim&\left(t_*^{2\alpha}\vee\frac1{\lambda_k^2}\right)\int_0^{T-t_*}(T-\tau)^{2H-3}d\tau\\
\lesssim&\left(t_*^{2\alpha+2H-2}\vee\frac1{\lambda_k^2t_*^{2-2H}}\right).
\end{align*}

For $K_3(T)$, since $0<T-u<T-\tau<t_*$ for $T-t_*<\tau<u<T$,
\begin{align*}
K_3(T)\lesssim&\int_{T-t_*}^T\left[\int_\tau^T\left[\int_{T-u}^{T-\tau}r^{
\alpha-2}dr\right]\left(\frac
u\tau\right)^{H-\frac12}(u-\tau)^{H-\frac32}du\right]^2d\tau\\
\lesssim&\int_{T-t_*}^T\left[\int_\tau^T(T-u)^{\alpha-2}(u-\tau)^{H-\frac12}\left(\frac u\tau\right)^{H-\frac12}du\right]^2d\tau\\
\lesssim&\int_{T-t_*}^T\left[\int_0^{T-\tau}(T-\tau-r)^{\alpha-2}r^{H-\frac12}
dr\right]^2d\tau\\
\lesssim&\int_{T-t_*}^T(T-\tau)^{2\alpha+2H-3}d\tau
\lesssim t_*^{2\alpha+2H-2},
\end{align*}
where we have used the condition $\alpha+H>1$ again.

Combining the above estimates, we conclude that
\begin{equation}\label{IIITe}
I_3(T)\lesssim\left(t_*^{2\alpha+2H-2}\vee\frac1{\lambda_k^2t_*^{2-2H}}\right).
\end{equation}

Finally, it follows from (\ref{ITe})--(\ref{IIITe}) that we obtain
\begin{align}\label{eq:IllH012}
&\mathbb{E}\left|\int_0^T(T-\tau)^{\alpha-1}E_{\alpha,\alpha}(-\lambda_k(T-\tau)^{\alpha})dB^H(\tau)\right|^2\nonumber\\
\lesssim&\max\left\{t_*^{2\alpha+2H-2}, \lambda_k^{-2}t_*^{2-2H}, \lambda_k^{-2}, t_*^{2H}, t_*\right\}.
\end{align}

\subsection{The case $H\in(\frac{1}{2},1)$}

Set $\tilde{p}=T-p, \tilde{q}=T-q$. From (\ref{Ivar}) and (\ref{V121}), we have
\begin{align}\label{vTe}
&\mathbb{E}\left|\int_0^T(T-\tau)^{\alpha-1}E_{\alpha,\alpha}(-\lambda_k(T-\tau)^{\alpha})dB^H(\tau)\right|^2\nonumber\\
=&\alpha_H\int_0^T\int_0^T(T-p)^{\alpha-1}E_{\alpha,\alpha}(-\lambda_k(T-p)^{\alpha})(T-q)^{\alpha-1}E_{\alpha,\alpha}(-\lambda_k(T-q)^{\alpha})|p-q|^{2H-2}dpdq\nonumber\\
=&\alpha_H\int_0^T\int_0^T\tilde{p}^{\alpha-1}E_{\alpha,\alpha}(-\lambda_k\tilde{p}^{\alpha})\tilde{q}^{\alpha-1}E_{\alpha,\alpha}(-\lambda_k\tilde{q}^{\alpha})|\tilde{q}-\tilde{p}|^{2H-2}d\tilde{p}d\tilde{q}\nonumber\\
=&\alpha_H\left(\int_0^{t_*}\int_0^{t_*}+\int_{t_*}^T\int_{t_*}^T+\int_0^{t_*}\int_{t_*}^T+\int_{t_*}^T\int_0^{t_*}\right)\tilde{p}^{\alpha-1}E_{\alpha,\alpha}(-\lambda_k\tilde{p}^{\alpha})\tilde{q}^{\alpha-1}E_{\alpha,\alpha}(-\lambda_k\tilde{q}^{\alpha})|\tilde{q}-\tilde{p}|^{2H-2}d\tilde{p}d\tilde{q}\nonumber\\
=:&\alpha_H(M_1(T)+M_2(T)+M_3(T)+M_4(T)).
\end{align}
We choose $t_*$ as that in (\ref{t*}). It is easy to see that
$M_3(T)=M_4(T)$. Then we only need to discuss $M_j(T), j=1, 2, 3.$

For the term $M_1(T)$, we can use the same analysis in Subsection
\ref{subsec3.2} to get
\begin{align}\label{M1T}
M_1(T)=&\int_0^{t_*}\int_0^{t_*}\tilde{p}^{\alpha-1}E_{\alpha,\alpha}
(-\lambda_k\tilde{p}^{\alpha})\tilde{q}^{\alpha-1}E_{\alpha,\alpha}
(-\lambda_k\tilde{q}^{\alpha})|\tilde{q}-\tilde{p}|^{2H-2}d\tilde{p}d\tilde{q}
\nonumber\\
\lesssim&\int_0^{t_*}\int_0^{t_*}\tilde{p}^{\alpha-1}\tilde{q}^{\alpha-1}|\tilde{q}-\tilde{p}|^{2H-2}d\tilde{p}d\tilde{q}\qquad(\text{noting}\, \alpha>0)\nonumber\\
\lesssim&t_*^{2\alpha+2H-2}.
\end{align}
For the term $M_2(T)$, it is easy to verify that
\begin{align}\label{M2T}
M_2(T)=&\int_{t_*}^T\int_{t_*}^T\tilde{p}^{\alpha-1}E_{\alpha,\alpha}(-\lambda_k\tilde{p}^{\alpha})\tilde{q}^{\alpha-1}E_{\alpha,\alpha}(-\lambda_k\tilde{q}^{\alpha})|\tilde{q}-\tilde{p}|^{2H-2}d\tilde{p}d\tilde{q}\nonumber\\
\lesssim&\frac{1}{\lambda_k^2}\int_{t_*}^T\int_{t_*}^T\frac{1}{\tilde{p}}\frac{1}{\tilde{q}}|\tilde{q}-\tilde{p}|^{2H-2}d\tilde{p}d\tilde{q}\nonumber\\
\leq&\frac{1}{\lambda_k^2t_*^2}\int_{t_*}^T\int_{t_*}^T|\tilde{q}-\tilde{p}|^{2H-2}d\tilde{p}d\tilde{q}\nonumber\\
=&\frac{1}{\lambda_k^2t_*^2}\left(\int_{t_*}^T\int_{t_*}^{\tilde{q}}(\tilde{q}-\tilde{p})^{2H-2}d\tilde{p}d\tilde{q}+\int_{t_*}^T\int^T_{\tilde{q}}(\tilde{p}-\tilde{q})^{2H-2}d\tilde{p}d\tilde{q}\right)\nonumber\\
=&\frac{2}{\lambda_k^2t_*^2}\int_{t_*}^T\int_{t_*}^{\tilde{q}}(\tilde{q}-\tilde{p})^{2H-2}d\tilde{p}d\tilde{q}\qquad (\text{noting}\, H>\frac12)\nonumber\\
=&\frac{2}{\lambda_k^2t_*^2}\int_{t_*}^T\frac{(\tilde{q}-t_*)^{2H-1}}{2H-1}d\tilde{q}\nonumber\\
=&\frac{2}{\lambda_k^2t_*^2}\frac{(T-t_*)^{2H}}{2H(2H-1)}
\lesssim \frac{1}{\lambda_k^2t_*^2}.
\end{align}
For the term $M_3(T)$, we may similarly have
\begin{align}\label{M3T}
M_3(T)=&\int_0^{t_*}\int_{t_*}^T\tilde{p}^{\alpha-1}E_{\alpha,\alpha}
(-\lambda_k\tilde{p}^{\alpha})\tilde{q}^{\alpha-1}E_{\alpha,\alpha}
(-\lambda_k\tilde{q}^{\alpha})|\tilde{q}-\tilde{p}|^{2H-2}d\tilde{p}d\tilde{q}
\nonumber\\
\lesssim&\frac{1}{\lambda_k}\int_0^{t_*}\int_{t_*}^T\frac{1}{\tilde{p}}\tilde{q}^{\alpha-1}(\tilde{p}-\tilde{q})^{2H-2}d\tilde{p}d\tilde{q}\nonumber\\
\leq&\frac{1}{\lambda_kt_*}\int_0^{t_*}\int_{t_*}^T\tilde{q}^{\alpha-1}(\tilde{p}-\tilde{q})^{2H-2}d\tilde{p}d\tilde{q}\nonumber\\
=&\frac{1}{\lambda_kt_*}\int_0^{t_*}\tilde{q}^{\alpha-1}\frac{(T-\tilde{q})^{2H-1}-(t_*-\tilde{q})^{2H-1}}{2H-1}d\tilde{q}\qquad (\text{noting}\, H>\frac12)\nonumber\\
\lesssim&\frac{1}{\lambda_kt_*}\int_0^{t_*}\tilde{q}^{\alpha-1}(T-\tilde{q})^{2H-1}d\tilde{q}\nonumber\\
\lesssim&\frac{1}{\lambda_kt_*}\int_0^{t_*}\tilde{q}^{\alpha-1}d\tilde{q}
\lesssim \frac{1}{\lambda_k}t_*^{\alpha-1}.
\end{align}

It follows from (\ref{vTe})--(\ref{M3T}) that we obtain the estimate
\begin{align}\label{eq:IllH121}
\mathbb{E}\left|\int_0^T(T-\tau)^{\alpha-1}E_{\alpha,\alpha}(-\lambda_k(T-\tau)^
{\alpha})dB^H(\tau)\right|^2
\lesssim \max\left\{t_*^{2\alpha+2H-2},\frac{1}{\lambda_k^2t_*^2},\frac{1}{
\lambda_k}t_*^{\alpha-1}\right\},
\end{align}
which is crucial to explain the instability of the inverse problem.

\subsection{Instability}

Based on the analysis above, we can obtain the following theorem which shows
that it is unstable to reconstruct $f$ and $|g|$.

\begin{theorem} The problem of recovering the source terms $f$ and $|g|$ is
unstable. Moreover, the following estimates hold
\begin{equation}\label{eq:ille1}
\left|\int_0^T(T-\tau)^{\alpha-1}E_{\alpha,\alpha}(-\lambda_k(T-\tau)^{\alpha}
)h(\tau)d\tau\right|\lesssim\lambda_k^{-1}
\end{equation}
and
\begin{equation}\label{eq:ille2}
\mathbb{E}\left|\int_0^T(T-\tau)^{\alpha-1}E_{\alpha,\alpha}(-\lambda_k(T-\tau)^
{\alpha})dB^H(\tau)\right|^2\lesssim\lambda_k^{-\beta},
\end{equation}
where
\begin{equation*}
\beta=\begin{cases}
\min\left\{2\gamma(\alpha+H-1), 2-2\gamma(H-1), 2H, \gamma\right\}, & 
0<H<\frac{1}{2},\\
\min\left\{2\gamma(\alpha+H-1), 2(1-\gamma), 1-\gamma(1-\alpha)\right\}, & 
\frac{1}{2}<H<1,
\end{cases}\quad 0<\gamma<1,\,\alpha+H>1
\end{equation*}
and
\[
 \beta=\min\left\{\gamma(2\alpha-1), 1-\gamma\right\},\quad
H=\frac{1}{2},\,0<\gamma<1,\,\alpha>\frac{1}{2}.
\]
\end{theorem}

\begin{proof}
For (\ref{eq:ille1}), one can refer to \cite[Lemma 4.4]{Niu+2018}. For
(\ref{eq:ille2}), one can obtain it by choosing $t_*=\lambda_k^{-\gamma},
0<\gamma<1$ in (\ref{eq:IllH012}) and (\ref{eq:IllH121}). Here the case
$H=\frac{1}{2}$ can be seen in \cite[Lemma 4.4]{Niu+2018}.
For $\alpha=1$, one can use $e^{-x}<\frac{1}{1+x}, x\geq0$ to obtain the same
results. Since $\lambda_k\to\infty$ as $k\to\infty$, the instability follows
easily from the estimates \eqref{eq:ille1}--\eqref{eq:ille2} and the
reconstruction formulas
\eqref{Iexp}--\eqref{Ivar}.
\end{proof}

\section{Conclusion}

In this paper, we have studied an inverse random source problem for the
time fractional diffusion equation driven by fractional Brownian motions. By the
analysis, we deduce the relationship of the time fractional order $\alpha$ and
the Hurst index $H$ in the fractional Brownian motion to ensure that the
solution is well-defined for the stochastic time fractional diffusion equation.
We show that the direct problem is well-posed when $\alpha+H>1$ and the inverse
source problem has a unique solution. But the inverse problem is ill-posed in
the sense that a small deviation of the data may lead to a huge error in the
reconstruction.

There are a few related interesting observation. First, if the Laplacian
operator is also fractional, the method can be directly applied and all the
results can be similarly proved. Second, for $1<\alpha\leq2$, the
direct problem can be shown to be well-posed since  Lemma \ref{MLinq} is still
valid. However, the inverse problem may not have a unique solution. The reason
is that Lemma \ref{Epositive} is not true any more for $1<\alpha\leq2$.
Finally, we mention that the numerics needs to be investigated. Clearly, some
regularization techniques are indispensable in order to suppress the
instability of the inverse problem. Another challenge is to how to
compute the integrals efficiently and accurately. We will report the numerical
results elsewhere in the future.

\appendix

\section{Fractional Brownian motion}

In the appendix, we briefly introduce the fractional Brownian motion (fBm) and
present some results which are used in this work.

\subsection{Definition and H\"older continuity}

A one dimensional fractional Brownian motion (fBm) $B^H$ with the Hurst
parameter $H\in(0,1)$ is a centered Gaussian process (i.e., $B^H(0)=0$)
determined by its covariance function
\begin{equation*}
R_H(t,s)=\mathbb{E}[B^H(t)B^H(s)]=\frac{1}{2}\left(t^{2H}+s^{2H}-|t-s|^{2H}\right)
\end{equation*}
for any $s,t\ge0$. In particular, if $H=\frac12$, $B^H$ turns to be the standard
Brownian motion, which is usually denoted by $W$, with covariance function
$R_H(t,s)=t\wedge s$.

The increments of fBms satisfies
\begin{equation*}
\mathbb{E}\left[\left(B^H(t)-B^H(s)\right)\left(B^H(s)-B^H(r)\right)\right]=\frac{1}{2}\left[(t-r)^{2H}-(t-s)^{2H}-(r-s)^{2H}\right]
\end{equation*}
and
\begin{equation*}
\mathbb{E}\left[\left(B^H(t)-B^H(s)\right)^2\right]=(t-s)^{2H}
\end{equation*}
for any $0<r<s<t$.
It then indicates that the increments of $B^H$ in disjoint intervals are
linearly dependent except for the case $H=\frac{1}{2}$, and the increments are
stationary since its moment depends only on the length of the interval.

Based on the moment estimates and the Kolmogorov continuity criterion, it holds
for any $\epsilon>0$ and $s,t\in[0,T]$ that
\[
|B^H(t)-B^H(s)|\le C|t-s|^{H-\epsilon}
\]
almost surely with constant $C$ depending on $\epsilon$ and $T$. That is, $H$ represents the regularity of $B^H$: the trajectories of fractional Brownian motion $B^H$ with Hurst parameter $H\in(0,1)$ are $(H-\epsilon)$-H\"older continuous.

\subsection{Representation of fBm and integration}

For a fractional Brownian motion $B^H$ with $H\in(0,1)$, it has the following Wiener integral representation
\[
B^H(t)=\int_0^tK_H(t,s)dW(s)
\]
with $K_H$ being a square integrable kernel and $W$ being the standard Brownian
motion (i.e., $H=\frac{1}{2}$).

For a fixed interval $[0,T]$, denote by $\mathcal{E}$ the space of step functions on $[0,T]$ and by $\mathcal{H}$ the closure of $\mathcal{E}$ with respect to the product
\[
\langle{\bf 1}_{[0,t]},{\bf 1}_{[0,s]}\rangle_{\mathcal{H}}=R_H(t,s),
\]
where ${\bf 1}_{[0,t]},{\bf 1}_{[0,s]}$ are the characteristic functions.
Define the linear operator $K^*_{H,T}:\mathcal{E}\to L^2(0,T)$ by
\begin{align}\label{eq:KH*}
(K_{H,T}^*\psi)(s)=K_H(T,s)\psi(s)+\int_s^T(\psi(u)-\psi(s))\frac{
\partial K_H(u,s)}{\partial u}du,
\end{align}
where
\[
\frac{\partial K_H(u,s)}{\partial u}=c_H\left(\frac us\right)^{H-\frac12}(u-s)^{H-\frac32}
\]
and $c_H$ is a constant given below depending on $H$. Then $K^*_{H,T}$ is an isometry from $\mathcal{E}$ to $L^2(0,T)$ (see e.g. \cite{Nualart+2006,Tindel+2003}), and the integral with respect to $B^H$ can be defined for functions $\varphi$ satisfying
\[
\|\psi\|_{|\mathcal{H}|}^2:=\langle
\psi,\psi\rangle_{\mathcal{H}}<\infty,
\]
and (see e.g. \cite{Nualart+2006,Tindel+2003})
\begin{align*}
\int_0^t\psi(s)dB^H(s)=&\int_0^T\psi(s){\bf
1}_{[0,t]}(s)dB^H(s)=\int_0^T[K_{H,T}^*(\psi{\bf 1}_{[0,t]})](s)dW(s)
\end{align*}
for any $t\in[0,T]$.
Hence, according to the It\^o isometry,
\begin{align}\label{eq:inte}
\mathbb{E}\left[\int_0^t\psi(s)dB^H(s)\int_0^t\phi(s)dB^H(s)\right]
=\langle K_{H,T}^*(\psi{\bf 1}_{[0,t]}),K_{H,T}^*(\phi{\bf
1}_{[0,t]})\rangle_{L^2(0,T)}.
\end{align}

\subsubsection{The case $H\in(\frac{1}{2},1)$}

For the case $H\in(\frac{1}{2},1)$, the covariance function $R_H$ of $B^H$ satisfies
\begin{align*}
R_H(t,s)=&\alpha_H\int_0^t\int_0^s|r-u|^{2H-2}dudr\\
=&\alpha_H\int_0^T\int_0^T{\bf 1}_{[0,t]}(r){\bf 1}_{[0,s]}(u)|r-u|^{2H-2}dudr
\end{align*}
with $\alpha_H=H(2H-1)$. The
square integrable kernel has form
\[
K_H(t,s)=c_H\int_s^t\left(\frac us\right)^{H-\frac12}(u-s)^{H-\frac32}du
\]
with $c_H=\left(\frac{\alpha_H}{\beta(2-2H,H-\frac12)}\right)^{\frac12}$ such that
\begin{align}\label{eq:chara}
\langle{\bf 1}_{[0,t]},{\bf 1}_{[0,s]}\rangle_{\mathcal{H}}=R_H(t,s)=&\alpha_H\int_0^T\int_0^T{\bf 1}_{[0,t]}(r){\bf 1}_{[0,s]}(u)|r-u|^{2H-2}dudr\nonumber\\
=&\int_0^T{\bf 1}_{[0,t]}(u){\bf 1}_{[0,s]}(u)K_H(t,u)K_H(s,u)du,
\end{align}
and $K_{H,T}^*$ in \eqref{eq:KH*} turns to be
\[
(K^*_{H,T}\psi)(s)=\int_s^T\psi(u)\frac{\partial K_H(u,s)}{\partial u}du.
\]
By noting that
\[
(K^*_{H,T}{\bf 1}_{[0,t]})(s)=\int_s^T{\bf 1}_{[0,t]}(u)\frac{\partial K_H(u,s)}{\partial u}du={\bf 1}_{[0,t]}(s)\int_s^t\frac{\partial K_H(u,s)}{\partial u}du={\bf 1}_{[0,t]}(s)K_H(t,s),
\]
one get
\begin{align*}
\langle{\bf 1}_{[0,t]},{\bf 1}_{[0,s]}\rangle_{\mathcal{H}}=&\int_0^T{\bf 1}_{[0,t]}(u){\bf 1}_{[0,s]}(u)K_H(t,u)K_H(s,u)du\\
=&\int_0^T(K^*_{H,T}{\bf 1}_{[0,t]})(u)(K^*_{H,T}{\bf 1}_{[0,s]})(u)du\\
=&\langle K_{H,T}^*{\bf 1}_{[0,t]},K_{H,T}^*{\bf 1}_{[0,s]}\rangle_{L^2(0,T)}.
\end{align*}

In this case, \eqref{eq:inte} can be calculated as follows
\begin{align}\label{V121}
&\mathbb{E}\left[\int_0^t\psi(s)dB^H(s)\int_0^t\phi(s)dB^H(s)\right]
\nonumber\\
=&\langle K_{H,T}^*(\psi{\bf 1}_{[0,t]}),K_{H,T}^*(\phi{\bf
1}_{[0,t]})\rangle_{L^2(0,T)}\nonumber\\
=&\langle\psi{\bf 1}_{[0,t]},\phi{\bf
1}_{[0,t]}\rangle_{\mathcal{H}}\nonumber\\
=&\alpha_H\int_0^t\int_0^t\psi(r)\phi(u)|r-u|^{2H-2}dudr
\end{align}
according to \eqref{eq:chara},
which is used in (\ref{eq:moment}).

\subsubsection{The case $H\in(0,\frac{1}{2})$}

If the trajectories of $B^H$ is less regular than the case above with $H\in(0,\frac12)$, the square integrable kernel $K_H$ has the following form instead
\begin{align}\label{KH012}
K_H(t,s)=c_H\left[\left(\frac ts\right)^{H-\frac12}(t-s)^{H-\frac12}-\left(H-\frac12\right)s^{\frac12-H}\int_s^tu^{H-\frac32}(u-s)^{H-\frac12}du\right]
\end{align}
with $c_H=\left(\frac{2H}{(1-2H)\beta(1-2H,H+\frac12)}\right)^{\frac12}$ such that
\[
R_H(t,s)=\int^{t\wedge s}_0K_H(t,u)K_H(s,u)du
\]
similar to \eqref{eq:chara}. Utilizing the fact (see \cite{Tindel+2003})
\[
[K_{H,T}^*(\psi{\bf 1}_{[0,t]})](s)=[(K_{H,t}^*\psi)(s)]{\bf
1}_{[0,t]}(s),\quad\forall~t\in[0,T],
\]
where $K_{H,t}^*$ is defined in \eqref{eq:KH*}, we may rewrite
\eqref{eq:inte} into
\begin{align}\label{V012}
\mathbb{E}\left[\int_0^t\psi(s)dB^H(s)\int_0^t\phi(s)dB^H(s)\right]
=&\langle K_{H,T}^*(\psi{\bf 1}_{[0,t]}),K_{H,T}^*(\phi{\bf
1}_{[0,t]})\rangle_{L^2(0,T)}\\
=&\langle K_{H,t}^*\psi,K_{H,t}^*\phi\rangle_{L^2(0,t)},
\end{align}
which is used in Subsection \ref{subsec3.2} and (\ref{vTe}).

\end{document}